\documentclass[journal]{IEEEtran}

\usepackage{mathtools}
\usepackage{amsthm, amsmath, amssymb, amsfonts, dsfont}
\usepackage{cite}   
\usepackage{graphicx}
\usepackage{subfig}
\usepackage{url}

\usepackage{tikz}
\usetikzlibrary{shapes,arrows}
\tikzset{%
  input/.style    = {draw, line width=0.45mm, rectangle, minimum height = 3em,
    minimum width = 3em},
  block/.style    = {draw, thick, rectangle, minimum height = 3em,
    minimum width = 3em},
  oper/.style      = {draw, circle, node distance = 2cm}, 
}

\usepackage[ruled, norelsize]{algorithm2e}
\makeatletter
\newcommand{\removelatexerror}{\let\@latex@error\@gobble}
\makeatother

\newcommand{\re}{\mathbb{R}}

\newcommand{\eps}{\varepsilon}

\renewcommand{\eps}{\varepsilon}
\renewcommand{\phi}{\varphi}
\renewcommand{\div}{\mathrm{\div}}

\newcommand{\Hc}{\mathcal{H}}

\newcommand{\Nc}{\mathcal{N}}
\newcommand{\Oc}{\mathcal{O}}

\newcommand{\Eb}{\mathbb{E}}

\newcommand{\Pb}{\mathbb{P}}

\newcommand{\Var}{\mathrm{Var}}

\newcommand{\defeq}{\vcentcolon=}

\providecommand{\abs}[1]{\left\lvert#1\right\rvert}
\providecommand{\norm}[1]{\left\lVert#1\right\rVert}
\providecommand{\set}[1]{\left\{#1\right\}}

\def\R{\mathbb{R}}

\def\N{\mathbb{N}}

\newcommand{\mean}[1]{\mathbb{E}\left[{#1}\right]}
\newcommand{\var}[1]{\mathbb{V}ar\left[{#1}\right]}

\DeclareMathOperator{\ind}{\mathds{1}}

\DeclareMathOperator{\nnz}{nnz}
\DeclareMathOperator{\st}{\mbox{s.t.}}
\DeclareMathOperator{\supp}{supp}

\def\expander{\mathbb{E}_{k, \varepsilon, d}^{m\times n}}

\def\Sparse{\chi_k^n}

\DeclareMathOperator{\expandermodel}{GM}

\newcommand\simiid{\stackrel{\text{i.i.d.}}{\sim}}

\catcode`@=11

\theoremstyle{definition}
\newtheorem{theorem}{Theorem}[section]
\newtheorem{lemma}[theorem]{Lemma}

\newtheorem{definition}[theorem]{Definition}

\newtheorem*{remark}{Remark}

\graphicspath{{./figures/}{./figures/phase_transitions/}{./figures/rho_vs_noise/}{./figures/timing_fixed_delta/}{./figures/analytical_vs_empirical/}}

\begin{document}

%
\title{A robust parallel algorithm for combinatorial compressed sensing}
%
%

\author{
    \IEEEauthorblockN{Rodrigo~Mendoza-Smith\IEEEauthorrefmark{1}\IEEEauthorrefmark{2}, Jared~Tanner\IEEEauthorrefmark{1}\IEEEauthorrefmark{2}, and~Florian~Wechsung\IEEEauthorrefmark{1}}
    \IEEEauthorblockA{
    \\~\\\IEEEauthorrefmark{1} Mathematical Institute, University of Oxford, Oxford, OX2 6GG.
    }
    \IEEEauthorblockA{
    \\\IEEEauthorrefmark{2} Alan Turing Institute, British Library, London, NW1 2DB.
    }
\thanks{RMS acknowledges the support of CONACyT}
\thanks{This work was supported by The Alan Turing Institute under the EPSRC grant EP/N510129/1}
\thanks{Furthermore, it is based on work partially supported by the EPSRC Centre For Doctoral Training in Industrially Focused Mathematical Modelling (EP/L015803/1) in collaboration with PA Consulting Group}
}


%
%


%

\maketitle


\begin{abstract}
%
%
It was shown in \cite{Tanner:2015aa} that a vector $x \in \R^n$ with at most $k < n$ nonzeros can be recovered from an expander sketch $Ax$ in $\mathcal{O}(\nnz(A)\log k)$ operations via the Parallel-$\ell_0$ decoding algorithm, where $\nnz(A)$ denotes the number of nonzero entries in $A \in \R^{m \times n}$.
In this paper we present the Robust-$\ell_0$ decoding algorithm, which robustifies Parallel-$\ell_0$ when the sketch $Ax$ is corrupted by additive noise.
This robustness is achieved by approximating the asymptotic posterior distribution of values in the sketch given its corrupted measurements.
We provide analytic expressions that approximate these posteriors under the assumptions that the nonzero entries in the signal and the noise are drawn from continuous distributions.
Numerical experiments presented show that Robust-$\ell_0$ is superior to existing greedy and combinatorial compressed sensing algorithms in the presence of small to moderate signal-to-noise ratios in the setting of Gaussian signals and Gaussian additive noise.
\end{abstract}

\section{Introduction}
\label{sec:introduction}
%
%
%
%




\IEEEPARstart{C}{ompressed} sensing is a well studied method by which a sparse or compressible vector can be acquired by a number of measurements proportional to the number of its dominant entries \cite{Candes:2005aa}, \cite{Donoho:2006aa}.
To fix notation, let $\Sparse$ be the set of vectors in $\R^n$ that have at most $k$ nonzero entries, let $x \in \Sparse$ and let $A\in\R^{m\times n}$ be a matrix with $m < n$.
We will refer to $A$ as the measurement matrix, $x$ as the signal and $y=Ax$ as the measurements.
The goal of compressed sensing is to recover the sparsest, most parsimonious, $x \in \R^n$ from the measurements $y$ and the matrix $A$.
Letting $\|\cdot\|_0$ denote the number of non-zeros in $x$, the problem of finding $x$ can be written as
\begin{equation*}
	\min_{x\in\R^n} \norm{x}_0 \text{ s.t. } Ax = y.
\end{equation*}
%
%
Many algorithms have been developed to solve this problem or equivalent formulations and there are good theoretical results on when and how fast recovery of a signal is possible given certain types of measurement matrix $A$ and signal $x$.
These algorithms can be broadly categorized into convex optimization
based algorithms like those implemented in \cite{donoho2007sparselab,
  kim2007interior, figueiredo2007gradient, candes2005l1} and greedy
algorithms \cite{Blumensath:2008aa, Pati:1993aa, needell2009cosamp,
  blumensath2010normalized, Dai:2013aaSubspacePursuit, Blanchard:2015aaCGIHT, Cevher:2011aaALPS}, and were
designed and analysed for the setting of dense sensing matrices;
e.g. independent \mbox{(sub-)Gaussian} entries or randomly subsampled Fourier matrices.
For a more detailed introduction to compressed sensing see \cite{Foucart:2013aa}.

Here we extend an algorithm proposed in \cite{Tanner:2015aa} which can be used to recover exactly a sparse signal from its expander sketch (see Section \ref{sec:combinatorical-compressed-sensing} for details).
Specifically, \cite{Tanner:2015aa} proposed Parallel-$\ell_0$ (Algorithm \ref{alg:parallel-l0}), for noiseless combinatorial compressed sensing which is guaranteed to converge in $\mathcal{O}(\nnz(A)\log(k))$ where the sensing matrix $A$ is an expander matrix (Definition \ref{def:expander_matrix}) and the signal $x \in \Sparse$ is {\em dissociated} in the sense of Definition \ref{def:dissociated} or the signal is drawn independently of $A$.
For alternative combinatorial compressed sensing algorithms see, for example,
\cite{Sarvotham:2006aaSudocodes,Xu:2007aa, Berinde:2008aa, Berinde:2009aa, Jafarpour:2009aa, Ma:2014aa}.
We borrow notation from combinatorics and use the shorthands $[n]:=\{1, \dots, n\}$, $[n]^{(k)}:=\{S \subset [n]: |S| = k\}$ where $|S|$ denotes the cardinality of the set $S$, and $[n]^{(\leq k)} := \cup_{\ell \leq k} [n]^{(\ell)}$ for $n,k \in \N$ and $k < n$.
%
%

\begin{definition}[Expander matrices \cite{Tanner:2015aa}]
\label{def:expander_matrix}
The matrix $A \in \{0,1\}^{m \times n}$ is a $(k, \eps, d)$-expander matrix if $\sum_{i = 1}^m \ind_{|A_{i,j}|>0} = d$ for all $j \in [n]$ and
\begin{equation*}
    \Big|\Big\{ i\in [m] : \sum_{j \in S} \ind_{|A_{i,j}| > 0}\Big\}\Big| > (1 - \eps)d|S|
\end{equation*}
for all $S \in [n]^{(\leq k)}$.
We denote by $\expander$ the set of $(k, \eps,d)$-expander matrices of dimension $m \times n$.
\end{definition}
\begin{definition}[Dissociated signals \cite{Tanner:2015aa}]
\label{def:dissociated}
A signal $x \in \R^n$ is said to be dissociated if
\begin{equation*}
	\sum_{j\in S_1} x_j \neq \sum_{j\in S_2} x_j \qquad \forall S_1,\ S_2 \subset \supp(x) \text{ s.t. } S_1 \neq S_2.
\end{equation*}
\end{definition}
An example of (almost surely) dissociated signals are those drawn from a continuous distribution.
It is shown in \cite{Tanner:2015aa} that if $y = Ax $ is an expander sketch and $x \in \Sparse$ is dissociated, then there exists a subset $T \subset [n]$ such that, for each $j \in T$, $|\{i \in [m] : y_i = x_j\}|$ is bounded below by a positive constant depending on $d$ and $\eps$.
This guarantees that if $|\{i \in \supp(a_j) : y_i = y_{\ell}\}| > d/2$  then $x_j = y_{\ell}$.
%
%
Parallel-$\ell_0$ (Algorithm \ref{alg:parallel-l0}) implements this observation by letting $\hat x = 0$ and estimating the decrease in $\|y\|_0$ when performing the update $\hat x_j \leftarrow \hat x_j + y_{\ell}$.
We denote for $j\in [n]$ its neighbours by $\Nc(j) \defeq \{i\in[m] : |A_{ij}|>0\}$;
to estimate the decrease in $\|y\|_0$, Parallel-$\ell_0$ computes 
\begin{align}
\label{eq:ne_nz_no_noise_1}
n_e &\gets \left|\{ \ell \in \mathcal{N}(j) : y_i = y_{\ell} \}\right|,\\
\label{eq:ne_nz_no_noise_2}
n_z &\gets \left|\{ \ell \in \mathcal{N}(j) : y_{\ell} = 0\}\right|.
\end{align}
We extend their approach to the additive noise signal model of $\hat y = y + \eta$ with $y = Ax$ and $\eta \in \R^m$ by replacing \eqref{eq:ne_nz_no_noise_1}-\eqref{eq:ne_nz_no_noise_2} with scores estimating the distribution of $n_e$ and $n_z$, e.g. \eqref{eq:ne_nz_1}-\eqref{eq:ne_nz_2}.
That is, we follow a Bayesian approach to the computation of these scores and estimate:
\begin{enumerate}
\item The probability of $y_i = 0$ given that we observe $\hat y_i$.
\begin{equation}
\label{eq:pz}
p_z\left(\omega\right) := \Pb\left( y_i = 0 \mid \hat y_i = \omega\right).
\end{equation}
\item The probability of $y_{i_1} = y_{i_2}$ given that we observe $\hat y_{i_1} - \hat y_{i_2}$.
\begin{equation}
\label{eq:pe}
p_e\left(\omega\right) := \Pb\left(y_{i_1} = y_{i_2} \mid \hat y_{i_1} - \hat y_{i_2} = \omega\right)
\end{equation}
\end{enumerate}

Among our contributions are series approximations for \eqref{eq:pz}-\eqref{eq:pe} when the signals and measurements are generated according to the generating model given in Definition \ref{def:expander_model} and illustrated in Figure \ref{fig:expander_model}.
In what follows, we let $\mathbb{D}(\R)$ be the set of distributions supported on $\R$.
If $\mu \in \mathbb{D}(\R)$, we write $z \sim \mu$ to denote that $z$ was drawn according to the distribution $\mu$. We also use the notation $v_i \simiid \mu$ to denote that each $v_i$ is drawn independently at random from $\mu$. Finally, we use $U(S)$ to denote the uniform distribution over a set $S$.

\begin{definition}[Generating model $\expandermodel(n, m, k, d, \mu, \nu)$]
\label{def:expander_model}
Let $n, m, k, d \in \mathbb{N}$ be such that $k < m < n$ and $d \ll m$.
Let $\mu, \nu \in \mathbb{D}(\R)$.
Then, the problem $(A, \hat y)$ is drawn from the model $\expandermodel(n, m, k, d, \mu, \nu)$ if $A \in \{0,1\}^{m \times n}$ and $\hat y\in \R^m$ are such that
\begin{enumerate}
\item each column of $A$ has a support drawn uniformly at random from $[m]^{(d)}$;
\item $\supp(x)$ is drawn uniformly at random from $[n]^{(k)}$;
\item $x_j \simiid \mu$ for each $j \in \supp(x)$;
\item $\eta_i \simiid \nu$ for each $i \in [m]$;
\item $\eta_i$ is independent of $x_j$ for all $i\in [m],j\in \supp(x)$;
\item $y = Ax$ and $\hat y = y + \eta$.
\end{enumerate}
We write $(A, \hat y)\sim \expandermodel(n, m, k, d, \mu, \nu)$ to
denote problem instances drawn from this signal model.
\end{definition}
\begin{remark}
It is important to note that a matrix $A \in \R^{m \times n}$ generated under the model presented in Definition \ref{def:expander_model} and Figure \ref{fig:expander_model}, is a $(k, \eps, d)$-expander matrix with high probability, see \cite{Bah:2013aa}, \cite[Theorem 13.6]{Foucart:2013aa}.
\end{remark}

\begin{figure}

\begin{tikzpicture}[thick, scale=0.78, node distance=2cm, >=triangle 45, every node/.style={scale=0.78}]
\draw
node at (0, -10)[input, right=-3mm, name=variables, align=center]{$n, m, k, d \in \mathbb{N}$ \\ $\mu, \nu \in \mathbb{D}(\R)$}
    node [block, above of=variables] (xsupp) {$\supp(x) \sim U\left([n]^{(k)}\right)$}
    node [block, above of=xsupp] (xmu) {$x_j \simiid \mu$ $\forall j \in \supp(x)$}
    node [block, above of=xmu] (x) {$x \in \Sparse$};
    \draw[->](variables) -- node {} (xsupp);
    \draw[->](xsupp) -- node {} (xmu);
    \draw[->](xmu) -- node {} (x);
\draw [color=gray,thick](-1,-9) rectangle (2.75,-3);
\draw
    node [block, right of=xsupp, right=+3mm] (colsupp) {$\supp(A_j) \sim U\left([m]^{(d)}\right)$}
    node [block, above of=colsupp] (colones) {$A_{i,j} = 1$ $\forall i \in \supp(A_j)$}
    node [input, above of=colones] (matrix) {$A \in \{0, 1\}^{m \times n}$};
    \draw[->](colsupp) -- node {} (colones);
    \draw[->](colones) -- node {} (matrix);
    \draw [color=gray,thick](3,-9) rectangle (7.05,-3);
\draw
    node [block, right of=colones, right=+3mm] (noisevals) {$\eta_i \simiid \nu$ $\forall i \in [m]$}
    node [block, above of=noisevals] (noise) {$\eta \in \mathbb{R}^m$};
    \draw[->](noisevals) -- node {} (noise);
    \draw [color=gray,thick] (7.2,-7) rectangle (10.15,-3);
\draw
    node [oper, above of=x, right=+15mm] (prod) {$\cdot$}
    node [block, above of=prod] (y) {$y=Ax$};
    \draw[->](x) |- node {} (prod);
    \draw[->](matrix) |- node {} (prod);
    \draw[->](prod) -- node {} (y);
\draw
    node [oper, right of=y, right=+15mm] (sum) {$+$}
    node  [input, above of=sum] (yhat) {$\hat y = y + \eta$};
    \draw[->](noise) |- node {} (sum);
    \draw[->](y) -- node {} (sum);
    \draw[->](sum) -- node {} (yhat);
    \draw[->](variables) -| node {} (colsupp);
    \draw[->](variables) -| node {} (noisevals);
\end{tikzpicture}

\caption[Generating model for Robust-$\ell_0$]{Generating model $\expandermodel(n, m, k, d, \mu, \nu)$.}
\label{fig:expander_model}
\end{figure}
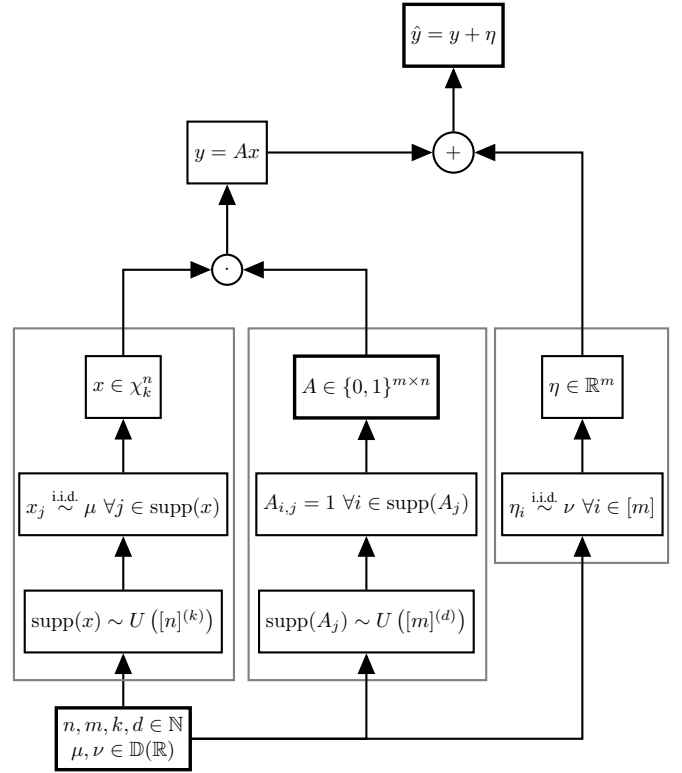

Moreover, the generating model in Definition \ref{def:expander_model} also allows us to define robust estimates for \eqref{eq:pz}-\eqref{eq:pe} for general noise and signal distributions and to any degree of accuracy under the assumption that these probability measures are available. 
From there we can define noisy analogues to the values $n_e$ and $n_z$ in \eqref{eq:ne_nz_no_noise_1}-\eqref{eq:ne_nz_no_noise_2} used in Parallel-$\ell_0$ but which are robust to additive noise.
The contributions of this work are two-fold: (i) to present principled ways to compute \eqref{eq:pz}-\eqref{eq:pe} in the case where the nonzeros in $\eta$ and $x$ are drawn from continuous probability distributions; (ii) to provide a variation of Parallel-$\ell_0$ that is robust to noise.
While other similar generating models can be considered using the techniques presented here, for ease of exposition and clarity, we restrict our description to this model.

\begin{theorem}[Probabilities for general signal and noise distributions]
\label{th:probabilities_general}
Let $\delta, \rho \in (0,1)$.
For each $n > 1$, let $m = \delta n$, $k = \rho m$ and $d < m$.
If $\mu, \nu \in \mathbb{D}(\R)$ and $(A, \hat y) \sim \expandermodel(n, m, k, d, \mu, \nu)$. Then as $n \rightarrow \infty$,
%
%
%
\begin{align}
  \label{eq:pz_general}
p_z\left(\omega\right) & \rightarrow \frac{\nu(\omega)}{\sum_{q\ge0} \frac{(d\rho)^q}{q!} (\nu \ast \mu_q)(\omega)},\\
  \label{eq:pe_general}
p_e\left( \omega\right) &\rightarrow \frac{\tilde\nu(\omega)}{\sum_{q\ge 0} \frac{(2d\rho)^q}{q!} (\tilde \nu \ast \bar\mu_q)(\omega)}.
\end{align}
Where $\mu_q, \bar \mu_q, \tilde \nu$ are probability measures constructed as in Definition \ref{def:measures}.
\end{theorem}

Equations \eqref{eq:pz_general} and \eqref{eq:pe_general} allows us to quantify the uncertainty associated with computing the score for Parallel-$\ell_0$ under the presence of additive noise.
Note that equations \eqref{eq:pz_general} and \eqref{eq:pe_general}
can be easily adapted to alternative generative models, such as where
the expected density of nonzeros per row varies, but for expository
clarity we restrict our discussion to this somewhat generic model.
It will be discussed in Section \ref{ssub:scaling-of-probabilities} that equations \eqref{eq:pz_general} and \eqref{eq:pe_general} should not be used directly, but instead be scaled by considering {\em normalised} functions $\breve p_e$ and $\breve p_z$ defined as,
\begin{equation}
\label{eq:breve}
\breve p_e(\omega) =  \frac{p_e(\omega)}{\max_s p_e(s)}, \;\;\;\breve p_z(\omega) =  \frac{ p_z(\omega)}{\max_s  p_z(s)}.
\end{equation}
In the most general case $n_e$ and $n_z$ can be written as the sum of individual scores $q_e$ and $q_z$ as follows
\begin{align}
\label{eq:ne_nz_1}
n_e &\gets  \sum_{\ell \in \mathcal{N}(j)} q_e\left(r_{i_1} - r_{i_2} \mid t\right)\\
\label{eq:ne_nz_2}
n_z &\gets \sum_{\ell \in \mathcal{N}(j)} q_z\left(r_i \mid t\right)
\end{align}
A confidence threshold $t >0$ needs to be given for some variants of our algorithm, so to simplify the exposition we include this parameter in the scores $q_e(\cdot \mid t)$ and $q_z(\cdot \mid t)$ regardless of whether it is used or not.
A summary of the score functions are given in Table \ref{table:scores}.

\begin{table}[!htbp]
\begin{equation*}
\begin{array}{|l||c | c| c|}
\hline
&&&\\
 & \mbox{Robust-$\ell_0$} & \mbox{Robust-$\ell_0$} & \\
 & \mbox{Continuous} & \mbox{Quantised} & \mbox{Parallel-$\ell_0$}\\
&&&\\
\hline
\hline
&&&\\
q_e(r_{i_1} - r_{i_2} \mid t) & \breve p_e(r_{i_1} - r_{i_2})& \ind_{\left\{ \breve p_e(r_{i_1} - r_{i_2}) \ge t \right\}} &\ind_{\left\{ r_{i_1} = r_{i_2}\right\}}\\
&&&\\
\hline
&&&\\
q_z(r_i \mid t) & \breve p_z(r_i)& \ind_{\left\{ \breve p_z(r_i) \ge 1 - t\right\}}&\ind_{\left\{ r_{i} =0\right\}}\\
&&&\\
\hline
\end{array}
\end{equation*}
\caption[Scores]{Extensions of scores used in Expander $\ell_0$-decoding to identify candidate updates to the sparse signal $\hat x$.}
\label{table:scores}
\end{table}

%
\begin{small}
  \removelatexerror
	\begin{algorithm}
   \KwData{$A \in \expander$; $y=Ax \in \R^m$ for $x\in \Sparse$; $\alpha \in (1, d]$; $\mu, \nu \in \mathbb{D}(\R)$; $c \in (0,1)$}
   \KwResult{ $\hat x \in \R^n$ $\st$ $ \hat x \approx x$}
    Estimate $\breve p_z = \breve p_z(d, k, m, n, \mu, \nu)$ as in \eqref{eq:breve}\;
    Estimate $\breve p_e = \breve p_e(d, k, m, n, \mu, \nu)$ as in \eqref{eq:breve}\;
    \eIf{{\footnotesize \tt quantised}}{
      Use {\em quantised} scores given in Table \ref{table:scores}.
    }{
      Use {\em continuous} scores given in Table \ref{table:scores}.
    }
    $\hat x \leftarrow 0$\;
    $\hat r \leftarrow y$\;
    $t \gets 1$\;
    \While{not converged and $t > 0$}{
		  $x' \gets \hat x$\;
      $r \gets \hat r$\;
		  \For{$j \in [n]$}{
        \For{$i \in \mathcal{N}(j)$}{
          \If{$1 - \breve p_z(r_{i}) \geq t$}{
            $n_e \gets \sum_{\ell \in \mathcal{N}(j)} q_e(r_i - r_{\ell} \mid t)$\;
            $n_z \gets \sum_{\ell \in \mathcal{N}(j)} q_z(r_\ell \mid t)$\;
            $\omega \gets \frac{1}{n_e}\sum_{\ell \in \mathcal{N}(j)} r_{\ell}q_e(r_i - r_{\ell} \mid t)$\;
		        \If{$\norm{r-\omega e_i}_1 \le \norm{r}_1$ and $n_e - n_z \geq \alpha$}{
		          $x'_j \gets x'_j + \omega$\;
            }
          }
        }
      }

			$x' \gets \Hc_k(x')$\;
			$r \gets y- A x'$\;
      $t \gets t - c$\;
		  \If{$\|r\|_1<\|\hat r\|_1$}{
		    $\hat x \gets x'$\;
		    $\hat r \gets r$\;
        \If{{\footnotesize \tt adaptive\_k}}{
          $k_0 \gets k - \sum_{j \in [n]} \breve p_z(\hat x_j)$\;
          $k_0 \gets \max(k_0, \lfloor \frac{m}{100}\rfloor)$\;
          Recompute $\breve p_z = \breve p_z(k_0, m, n, \mu, \nu)$\;
          Recompute $\breve p_e = \breve p_e(k_0, m, n, \mu, \nu)$\;
        }
      }
    }
  \caption{Robust-$\ell_0$}
  \label{alg:robust-l0}
	\end{algorithm}
\end{small}

\subsection{Outline of the manuscript}

The structure of this paper is as follows: Section \ref{sec:combinatorical-compressed-sensing} comprises a review of combinatorial
compressed sensing and the Parallel-$\ell_0$ decoding algorithm
extended here.
Robust-$\ell_0$ decoding and the associated scores \eqref{eq:ne_nz_1}-\eqref{eq:ne_nz_2} are presented in Section \ref{sec:l0-decoding-noisy}.
In Section \ref{sec:numerical-evidence} we present numerical experiments which demonstrate Robust-$\ell_0$ to perform superior to a number of leading greedy and combinatorial compressed sensing algorithms.

\section{Background: Combinatorial Compressed Sensing and $\ell_0$-decoding}
\label{sec:combinatorical-compressed-sensing}

As mentioned in the previous section, the branch of combinatorial compressed sensing measures $x \in \Sparse$ with the adjacency matrix of an expander graph.
These matrices are of very low complexity in terms of generation and storage, and also promise faster encoding and decoding than their dense counterparts, see Theorem \ref{th:l0-decoding-convergence}.
In this section we review the basic elements of expander graphs and combinatorial compressed sensing.

There have been various algorithms proposed to reconstruct a sparse vector $x$ from measurements $y=Ax$ when $A$ is an expander matrix, see \cite{Xu:2007aa}, \cite{Berinde:2008aa}, \cite{Jafarpour:2009aa} and \cite{Berinde:2009aa}; the work presented here starts with the Parallel-$\ell_0$ algorithm and to improve this algorithm by making it robust to noisy measurements, results in Robust-$\ell_0$ (Algorithm \ref{alg:robust-l0}).
The key observation for the Parallel-$\ell_0$ algorithm is given by the following Lemma.

\begin{small}
\removelatexerror
\begin{algorithm}
 \KwData{$A \in \expander$; $y=Ax \in \R^m$ for $x \in \Sparse$; $\alpha \in (1, d]$}
 \KwResult{ $\hat x \in \R^n$ $\st$ $\hat x = x$ }
 $\hat x \leftarrow 0$, $r \leftarrow y$\;
  \While{not converged}{
    \For{$j \in [n]$}{
      $u \leftarrow 0$\;
      \For{$i \in \mathcal{N}(j)$}{
        \If{ $r_{i} \neq 0$}{
          $n_e \gets \left|\{ \ell \in \mathcal{N}(j) : r_i = r_{\ell} \}\right|$\;
          $n_z \gets \left|\{ \ell \in \mathcal{N}(j) : r_{\ell} = 0\}\right|$\;
          \If{$n_e - n_z \geq \alpha$}{
            $u_j \gets r_i$\;
          }
        }
      }
    }
    $\hat x \leftarrow \hat x + u$\;
    $r \gets r - A\hat x$\;
}
\caption{Parallel-$\ell_0$ \cite{Tanner:2015aa}}
\label{alg:parallel-l0}
\end{algorithm}
\end{small}

\begin{lemma}
	Let $y=Ax$, $x$ dissociated, $A\in \expander$ with $\varepsilon<\frac{1}{4}$. Then there exists a nonempty set $T\subset [n]\times\re$ such that 
	\begin{equation*}
		\abs{\{i\in\Nc(j):y_i=\omega\}}\ge (1-2\varepsilon)d \;\; \forall (j, \omega) \in T
	\end{equation*}
	and for every tuple in $T$ that satisfies this property, we have $w=x_j$.
\end{lemma}
This means that at each iteration, if the residual $r$ is non-zero, i.e. if we have not yet found the correct $x$, then there is a set of entries in $x$ that we can change so that we reduce the number of non-zeros in $r$ by at least $|T|(1- 2 \varepsilon)d$.

\begin{theorem}[Convergence of Algorithm \ref{alg:parallel-l0} \cite{Tanner:2015aa}]\label{th:l0-decoding-convergence}
	Let $A\in \expander$ and let $\varepsilon\le \frac{1}{4}$, and $x\in\chi_k^n$ be dissociated.
	Then, Parallel-$\ell_0$ with $\alpha = (1-2 \varepsilon)d$ can recover $x$ from $y=Ax$ in $\Oc(\log k)$ iterations of complexity $\Oc(dn)$.
\end{theorem}

To put this result into context and show its applicability, we recall the remark after Definition \ref{def:expander_model}, stating that random matrices as considered in this work are indeed expander matrices with high probability.

We furthermore emphasize the fact that the algorithm is designed in a way that allows for massively parallel implementations.

%

\section{Main contributions: $\ell_0$-decoding for Noisy Measurements}
\label{sec:l0-decoding-noisy}

We now consider the case where the measurements $y$ are subject to additive noise, i.e. instead of $y=Ax$, we measure $\hat y = y + \eta$ where $\eta$ is a realization of a random variable with $\eta_i \sim \nu$.

Parallel-$\ell_0$ is not able to cope with additive noise, as it needs to make decisions whether a value in $\hat y$ is zero and whether two values in $\hat y$ are equal to each other. 
While for very small noise levels we could consider two values as equal if they are within a certain number of standard deviations, for larger noise levels the decision becomes more challenging.
A discussed in Section \ref{sec:introduction} we need to know $p_z(\hat y_i)$ and $p_e(\hat y_{i_1} - \hat y_{i_2})$ which correspond, respectively, to the probability of $y_i = 0$ given that we observe $\hat y_i$ and the probability of $y_{i_1} = y_{i_2}$ given that we observe $\hat y_{i_1} - \hat y_{i_2}$.
%
%
%
The functions $p_z$ and $p_e$ depend on the parameters fed into the generative model in Definition \ref{def:expander_model} and in particular on the distribution of $y$ and $\hat y$.
Hence,
since $y$ is a vector of sparse inner products we need to understand the limiting behaviour of sparse sums.

We include Definition \ref{def:measures} in order to remind the reader
of some actions on measures used in this manuscript as so as to be
relatively self contained.

%


\begin{definition}[Measures \cite{Klenke:2013aa}]
\label{def:measures}
Let $\mathcal{B}(\R)$ denote the Borel $\sigma$-algebra over $\R$.
If $E \in \mathcal{B}(\R)$ let $-E:= \{-x : x \in E\}$.
Let $\mu \in \mathbb{D}(\R)$, we define the following measures.
\begin{enumerate}
\item The $q$-convolution, 
\begin{align*}
    \mu_0(E) &= \delta_0(E) = \left\{\begin{array}{ll}
 1 & 0 \in E\\
 0 & 0 \notin E
 \end{array}\right., \forall E \in \mathcal{B}(\R)\\
\mu_1(E) &= \mu(E), \forall E \in \mathcal{B}(\R)\\
\mu_{q+1}(E) &= (\mu_{q} \ast \mu)(E), \forall E \in \mathcal{B}(\R), q \in \mathbb{N}
\end{align*}
\item The negative measure,
\begin{equation*}
\mu^{-}(E) = \mu(-E), \;\; \forall E \in \mathcal{B}(\R)
\end{equation*}
\item The symmetrized measure,
\begin{equation*}
\bar\mu(E) = \frac{\mu(E) + \mu(-E)}{2}, \;\; \forall E \in \mathcal{B}(\R)
\end{equation*}
\item The measure associated with the difference of two random variables,
\begin{equation*}
\tilde \mu(E) = (\mu \ast \mu^{-})(E), \;\; \forall E \in \mathcal{B}(\R).
\end{equation*}
\end{enumerate}
\end{definition}


\begin{lemma}[Limiting distribution for sparse sums of random variables\label{prop:1}]
\label{lemma:sparse_sums}
Let $p\in (0,1)$, let $\mu \in \mathbb{D}(\R)$ and let $\mu_q \in \mathbb{D}(\R)$ be its the $q$-fold convolution.
For each $n\ge 1$, let
\begin{equation}
s_n {:=} \sum_{j=1}^n b_j x_j
\end{equation}
be such that,
\begin{enumerate}
\item $x_j \simiid \mu$ for each $j \in [n]$,
\item $b_j \simiid \mathrm{Ber}(\frac{p_n}{n})$ for each $j \in [n]$ with $p_n \to p \in \R$ as $n \to \infty$.
\end{enumerate}
Then, as $n\to \infty$ it holds that $s_n \overset{(d)}{\to} s$
where
\begin{equation}
\label{eq:sparse_sums}
s\sim \exp(-p)\sum_{q \geq 0} \frac{p^q}{q!}\mu_q.
\end{equation}
\end{lemma}

\begin{proof}
Let $\psi_{s_n}(t)$ be the characteristic function of $s_n$.
Let $x \sim \mu$, then
\begin{align*}
\psi_{s_n} (t) &= \Eb\left[ \exp(it s_n)\right]\\
&= \prod_{j=1}^n \Eb\left[\exp(it b_j x_j)\right]\\
&= \left(\left(1-\frac{p_n}{n}\right) + \left(\frac{p_n}{n}\right)\psi_x(t)\right)^n\\
&= \left( 1 + \frac{p_n (\psi_x(t) - 1)}{n}\right)^n.
\end{align*}
%
Taking the limit $n\to \infty$ we see that
\begin{align}\label{eq:char_s}
\lim_{n \rightarrow \infty} \psi_{s_n}(t)=\exp(-p)\exp(p \psi_x(t)).
\end{align}
Letting $w_q = \sum_{j = 1}^q x_j$ it holds by the independence of $\{x_1, \dots, x_q\}$ that $w_q \sim \mu_q$ and
\begin{equation}
\left[\psi_x(t)\right]^q = \psi_{w_q}(t).
\end{equation}
Now, consider a random variable $z$ distributed according to
\begin{equation}
z \sim \exp(-p) \sum_{q\geq 0} \frac{p^q}{q!}\mu_q.
\end{equation}
The characteristic function of $z$ is given by
\begin{align}
\psi_z(t) &= \mathbb{E}\left[\exp(itz)\right] \nonumber \\
&=\exp(-p) \sum_{q\geq 0} \frac{p^q}{q!}\psi_{w_q}(t) \nonumber \\
&=\exp(-p) \sum_{q\geq 0} \frac{p^q}{q!}\left(\psi_{x}(t)\right)^q \nonumber \\
&=\exp(-p) \sum_{q\geq 0} \frac{\left(p \psi_x(t)\right)^q}{q!} \nonumber \\
\label{eq:char_z}
&=\exp(-p) \exp(p\psi_x(t)).
\end{align}
Therefore \eqref{eq:char_z} equals \eqref{eq:char_s}.
By L\'evy's continuity Theroem pointwise convergence of the characteristic functions implies weak convergence of the random variables (cf. \cite[Theorem 15.23]{Klenke:2013aa}) and hence the statement follows.
\end{proof}



\begin{theorem}[Distribution of $\hat y_i$ and $\hat y_{i_1} - \hat y_{i_2}$]
\label{thm:main}
Fix $\delta, \rho \in (0,1)$ and for $n\in \N$ let $m=\delta n$, $k=\rho m$ and $d\ll m$.
Furthermore, let $\mu$ and $\nu$ be measures and assume that $\hat y = y + \eta = Ax + \eta$ is drawn from the model $\mathrm{GM}(n, m, k, d, \mu, \nu)$.
Then as $n\to \infty$
\begin{equation*}
\hat y_i \overset{(d)}{\to} \hat y_i^* \;\;\;\mbox{ and } \;\;\;\hat y_{i_1} - \hat y_{i_2} \overset{(d)}{\to} \hat g^*
\end{equation*}
where
\begin{align}
\label{eq:dist_yhat_star}
\hat y_i^* &\sim \exp(-d\rho)\sum_{q\ge 0} \frac{(d\rho)^q}{q!} \nu \ast \mu_q, \\
\label{eq:dist_ghat_star}
\hat g^* &\sim \exp(-2d\rho)\sum_{q\ge 0} \frac{(2d\rho)^q}{q!} \tilde\nu \ast \bar\mu_q.
\end{align}
\end{theorem}


\begin{proof}
To show \eqref{eq:dist_yhat_star}, let $i \in [m]$ and
\begin{equation*}
\label{eq:sparse_sum_cs}
y_i = \sum_{j=1}^n A_{i,j} x_j.
\end{equation*}
By our assumptions on $A$ and $x$,
\begin{align*}
\Pb\left(A_{i,j}x_j \neq 0\right) &= \Pb\left( A_{i,j} \neq 0 \wedge x_j \neq 0 \right)\\
&= \Pb\left( A_{i,j} \neq 0\right) \Pb\left(x_j \neq 0 \right)\\
&= \frac{d}{m} \frac{k}{n}\\
&= \frac{d \rho}{n}
\end{align*}
where for two events $E_1$ and $E_2$ we let $E_1 \wedge E_2$ be the conjunction of the events.
Note that if $A_{i,j}x_j \neq 0$, then $j \in \supp(x)$ so $A_{i,j}x_j = x_j \sim \mu$.
Hence, letting $b_j \sim \mathrm{Ber}\left(\frac{d\rho}{n}\right)$,
\begin{equation*}
  y_i \overset{(d)}{=} \sum_{j=1}^n b_j x_j.
\end{equation*}
Invoking Lemma \ref{lemma:sparse_sums} with $p_n= p = d\rho$ we obtain that  $y_i \rightarrow y_i^*$ as $n \to \infty$ where
\begin{equation*}
\label{eq:dist_y_star}
y_i^* \sim \exp(-d\rho) \sum_{q\geq 0} \frac{(d\rho)^q}{q!}\mu_q.
\end{equation*}
By the independence of $y_i^*$ and $\eta_i$, since $\hat y_i = y_i + \eta_i$, the distribution of $\hat y_i^*$ is given by
\begin{equation}
\label{eq:dist_y_star_nu}
\hat y_i^* \sim \left(\exp(-d\rho) \sum_{q\geq 0} \frac{(d\rho)^q}{q!}\mu_q\right)\ast \nu.
\end{equation}
Equation \eqref{eq:dist_yhat_star} follows from \eqref{eq:dist_y_star_nu} and the distributivity of the convolution operator.

To show \eqref{eq:dist_ghat_star}, let $i_1, i_2 \in [m]$ be such that $i_1 \neq i_2$ and let
\begin{equation*}
y_{i_1} - y_{i_2} = \sum_{j=1}^n \left(A_{i_1, j} - A_{i_2, j}\right)x_j.
\end{equation*}
Similarly to the previous case, we compute
\begin{align*}
\Pb\left(\left(A_{i_1,j} - A_{i_2,j} \right) x_j \neq 0\right) &= \Pb\left( A_{i_1,j} - A_{i_2,j} \neq 0 \wedge x_j \neq 0 \right)\\
&= \Pb\left( A_{i_1,j} - A_{i_2, j} \neq 0\right) \Pb\left(x_j \neq 0 \right)\\
&= 2\frac{d}{m}\frac{(m-1) - (d-1)}{m-1} \frac{k}{n}\\
&= \frac{2d\rho}{n}\left(1 - o(1)\right)
\end{align*}
Note that if $(A_{i_1,j} - A_{i_2, j})x_j \neq 0$, then $j \in \supp(x)$ and $(A_{i_1, j} - A_{i_2,j})$ is either $+1$ with probability $\frac{1}{2}$ or $-1$ with probability $\frac{1}{2}$. Then,
Hence,
\begin{equation*}
(A_{i_1,j} - A_{i_2, j})x_j \sim \left\{\begin{array}{ll}
\mu & \mbox{with probability $\frac{1}{2}$},\\
\mu^{-} & \mbox{with probability $\frac{1}{2}$},
\end{array}\right.
\end{equation*}
then,
\begin{equation*}
(A_{i_1,j} - A_{i_2, j})x_j \sim \bar \mu.
\end{equation*}
Letting $b'_j \sim \mathrm{Ber}\left(\frac{2d\rho}{n}(1 - o(1))\right)$,
\begin{equation*}
y_{i_1} - y_{i_2} \overset{(d)}{=} \sum_{j=1}^n b_j'x_j.
\end{equation*}
Again, invoking Lemma \ref{lemma:sparse_sums} with $p_n = {2d\rho}(1 - o(1))$ we obtain that $p = 2d\rho$ and also that as $n \rightarrow \infty$, $y_{i_1} - y_{i_2} \to g^*$ with
\begin{equation}
\label{eq:dist_yy_star}
y_{i_1}^* - y_{i_2}^* \sim \exp(-2d\rho) \sum_{q\geq 0} \frac{(2d\rho)^q}{q!}\bar \mu_q.
\end{equation}
Therefore,
Given that $\eta_{i_1} - \eta_{i_2} \sim \nu \ast \nu^{-}$ and that
\[
\hat y_{i_1} - \hat y_{i_2} = \left(y_{i_1} - y_{i_2}\right) + \left(\eta_{i_1} - \eta_{i_2}\right),
\]
we convolve \eqref{eq:dist_yy_star} with $\nu \ast \nu^{-}$ to recover \eqref{eq:dist_ghat_star}.
\end{proof}

We are now ready to prove Theorem \ref{th:probabilities_general},

\begin{proof}[Proof. Theorem \ref{th:probabilities_general}]
Using Bayes rule we write
\begin{equation}
\label{eq:prob_being_zero-1}
\Pb(y_i = 0 \vert \hat y_i = \omega) =  \frac{\Pb( \hat y_i = \omega \wedge y_i = 0 )}{\Pb(\hat y_i = \omega)}.
\end{equation}
From \eqref{eq:dist_yhat_star} we can deduce that
\begin{equation}
\label{eq:pbz_num}
	\Pb( \hat y_i  = \omega\wedge y_i = 0 ) = \exp(-d\rho)\nu(\omega)
\end{equation}
and using equation \eqref{eq:dist_yhat_star} from Theorem \ref{thm:main} we obtain that as $n \to \infty$,
\begin{equation}
\label{eq:pbz_den}
\Pb(\hat y_i = \omega) \to \exp(-d\rho) \sum_{q\ge0} \frac{(d\rho)^q}{q!} (\nu \ast \mu_q)(\omega).
\end{equation}
Coupling \eqref{eq:prob_being_zero-1}, \eqref{eq:pbz_den} and \eqref{eq:pbz_num} yields \eqref{eq:pz_general}.

Again, by Bayes rule,
\begin{equation}
\label{eq:prob_being_equal-1}
\Pb\left(y_{i_1} = y_{i_2} \vert \hat y_{i_1} - \hat y_{i_2} = \omega\right) =  \frac{\Pb\left(y_{i_1} = y_{i_2} \wedge \hat y_{i_1} - \hat y_{i_2} = \omega\right)}{\Pb\left(\hat y_{i_1} - \hat y_{i_2} = \omega\right)}.
\end{equation}
Noting that $\hat y_{i_1} - \hat y_{i_2} = \left(y_{i_1} - y_{i_2}\right) + \left(\eta_{i_1} - \eta_{i_2}\right)$,
\begin{align}
\label{eq:pbe_num}
\Pb\left(y_{i_1} = y_{i_2} \wedge \hat y_{i_1} - \hat y_{i_2} = \omega\right) &= \tilde \nu (\omega)
\end{align}
By \eqref{eq:dist_yhat_star} from Theorem \ref{thm:main} we obtain that as $n \to \infty$,
\begin{equation}
\label{eq:pbe_den}
\Pb(\hat y_{i_1} - \hat y_{i_2} = \omega) \to \exp(-2d\rho)\sum_{q\ge 0} \frac{(2d\rho)^q}{q!} \tilde\nu \ast \bar\mu_q(\omega)
\end{equation}
Coupling \eqref{eq:prob_being_equal-1}, \eqref{eq:pbe_num}, and \eqref{eq:pbe_den} yields \eqref{eq:pe_general}.

\end{proof}

\subsection{Explicit formulas for the centred Gaussian case}

We further elucidate \eqref{eq:pz_general}-\eqref{eq:pe_general} from Theorem \ref{th:probabilities_general} in the case when $\mu$ and $\nu$ are Gaussian with mean zero, which are the distributions considered in Section \ref{sec:numerical-evidence}.
To this end, let $\mu=\Nc(0, \sigma_s^2)$ and $\nu = \Nc(0, \sigma_n^2)$. 
We observe that in this case $\bar \mu = \mu$, $\mu_q = \bar\mu_q = \Nc(0, q \sigma_s^2)$ and $\tilde \nu = \Nc(0, 2 \sigma_n^2)$.
Hence, denoting by $\varphi(\cdot \mid \sigma^2)$ the probability density function of a centred Gaussian random variable with variance $\sigma^2$, we obtain 
\begin{align}
\label{eq:pz_gaussian}
		p_z(\omega) &\to \frac{\varphi(\omega\mid \sigma_n^2)}{\sum_{q\ge 0} \frac{(d\rho)^q}{q!} \varphi(\omega\mid q\sigma_s^2+\sigma_n^2)}\\
\label{eq:pe_gaussian}
		p_e(\omega) &\to \frac{\varphi(\omega \mid 2\sigma_n^2)}{\sum_{q\ge 0} \frac{(2d\rho)^q}{q!} \varphi(\omega \mid q\sigma_s^2+2\sigma_n^2)}. 
\end{align}

\subsubsection{Estimating the tails in the Gaussian case}

We approximate the infinite sum in in the denominators of \eqref{eq:pz_gaussian} and \eqref{eq:pe_gaussian} by doing an approximation to the tail of this summation.
%

\begin{lemma}[Sums of centred density functions]
\label{lemma:sums_centred_df}
Let $\mu_i$ be the density function of a random varible with mean zero and variance $\sigma_i^2$ and let $\alpha_i > 0$ be such that $\sum_i \alpha_i = 1$.
Then, $\mu = \sum_i \alpha_i \mu_i$ is the density function of a random variable with mean zero and variance $\sum_i \alpha_i \sigma_i^2$.
\end{lemma}
%

\begin{proof}
Let $x \sim \mu$ and $x_i \sim \mu_i$ be such that $\mean{x_i} = 0$ and $\var{x_i} = \sigma_i^2$.
\begin{align*}
\var{x} &= \mathbb{E}[x^2]\\
&= \int \omega^2 \mu(\omega) d\omega\\
&= \int \omega^2 \left(\sum_i \alpha_i \mu_i(\omega)\right) d\omega\\
&= \sum_i \alpha_i \int \omega^2 \mu_i(\omega) d\omega\\
&= \sum_i \alpha_i \sigma_i^2
\end{align*}
\end{proof}

To simplify notation let $\sigma_q^2 = q\sigma_s^2 + \sigma_n$ for $q \in \mathbb{N} \cup \{0\}$.
Consider the series in the denominator of \eqref{eq:pbz_den}
\begin{equation*}
S(\omega):=\sum_{q=0}^{\ell} \frac{(d\rho)^q}{q!} \varphi(\omega \mid \sigma_q^2) + \sum_{q=\ell + 1}^{\infty} \frac{(d\rho)^q}{q!} \varphi(\omega \mid \sigma_q^2).
\end{equation*}
Let,
\begin{equation*}
R_z(\ell):= \exp(d\rho) - \sum_{q = 0}^{\ell} \frac{(d\rho)^q}{q!}
\end{equation*}
and write
\begin{align*}
S_a(\omega)&:=\sum_{q = 0}^{\ell} \frac{(d\rho)^q}{q!} \varphi(\omega \mid \sigma_q^2),\\
S_b(\omega)&:=\sum_{q = \ell + 1}^{\infty} \frac{(d\rho)^q}{q!} \varphi(\omega \mid \sigma_q^2).
\end{align*}
Note that $S_b/R_z$ satisfies the conditions of Lemma \ref{lemma:sums_centred_df} so it corresponds to the density function of a centred random variable with variance
\begin{align*}
\sigma_z^2 &= \frac{1}{R_z}\sum_{q = \ell + 1}^{\ell} \frac{(d\rho)^q}{q!}(q \sigma_s^2 + \sigma_n^2)\\
 &= \frac{\sigma_s^2 (d\rho) R_z(\ell -1) + \sigma_n^2 R_z(\ell)}{R_z(\ell)}
\end{align*}
Therefore,
\begin{equation}
\label{eq:pz_approx}
	p_z(\omega) \approx \frac{\varphi(\omega \mid \sigma_n^2)}{\sum_{q = 0}^\ell \frac{(d\rho)^q}{q!} \varphi(\omega\mid \sigma_q^2) + R_z(\ell) \varphi(\omega \mid \sigma_{z}^2)}.
\end{equation}

A similar argument shows that
\begin{equation}
\label{eq:pe_approx}
	p_e(\omega) \approx \frac{\varphi(\omega \mid 2\sigma_n^2)}{\sum_{q = 0}^\ell \frac{(2d\rho)^q}{q!} \varphi(\omega\mid \sigma_{q,e}^2) + R_e(\ell) \varphi(\omega \mid \sigma_{e}^2)}.
\end{equation}

Where $\sigma_{q,e}^2 = q\sigma_s^2 + \sigma_n^2$, and
\begin{align*}
R_e(\ell) &= \exp(2d\rho) - \sum_{q = 0}^{\ell} \frac{(d\rho)^q}{q!},\\
\sigma_e^2&= \frac{\sigma_s^2 (2d\rho) R_e(\ell -1) + 2\sigma_n^2 R_e(\ell)}{R_e(\ell)}.
\end{align*}

\subsection{Comparison with empirical probabilities}
We test the approximations given in \eqref{eq:pz_approx} and \eqref{eq:pe_approx} by randomly generating $\hat y$ and $y$ according the generating models $\expandermodel(n, \delta n, \rho \delta n, 7, \mathcal{N}(0,1), \mathcal{N}(0, \sigma^2))$, for $\delta=0.3$, $\rho\in\{0.1, 0.3\}$ and $\sigma\in \{10^{-3}, 10^{-2}\}$.
The results can be seen in Figure \ref{fig:empirical-vs-analytical-probabilities}.
%

\begin{figure}
	\centering
	\subfloat[~]{\includegraphics[width=0.49\linewidth]{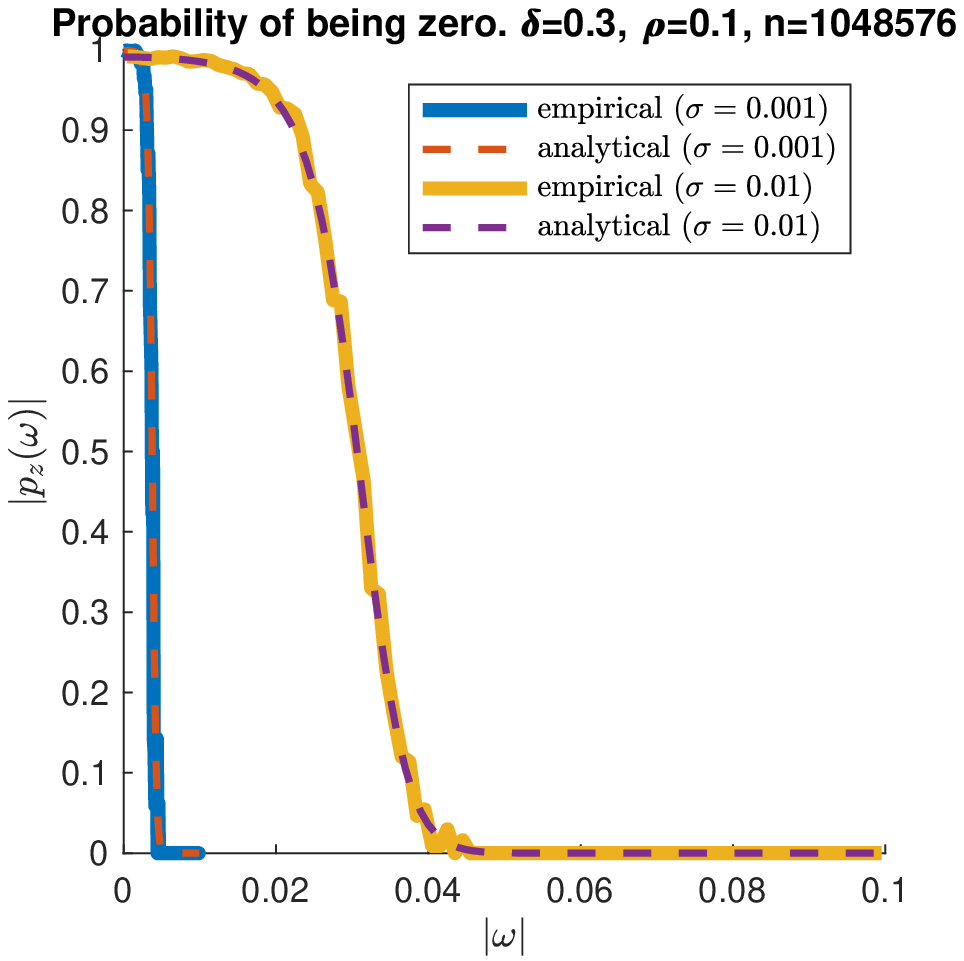}\label{fig:ana-vs-emp-a}}
	\subfloat[~]{\includegraphics[width=0.49\linewidth]{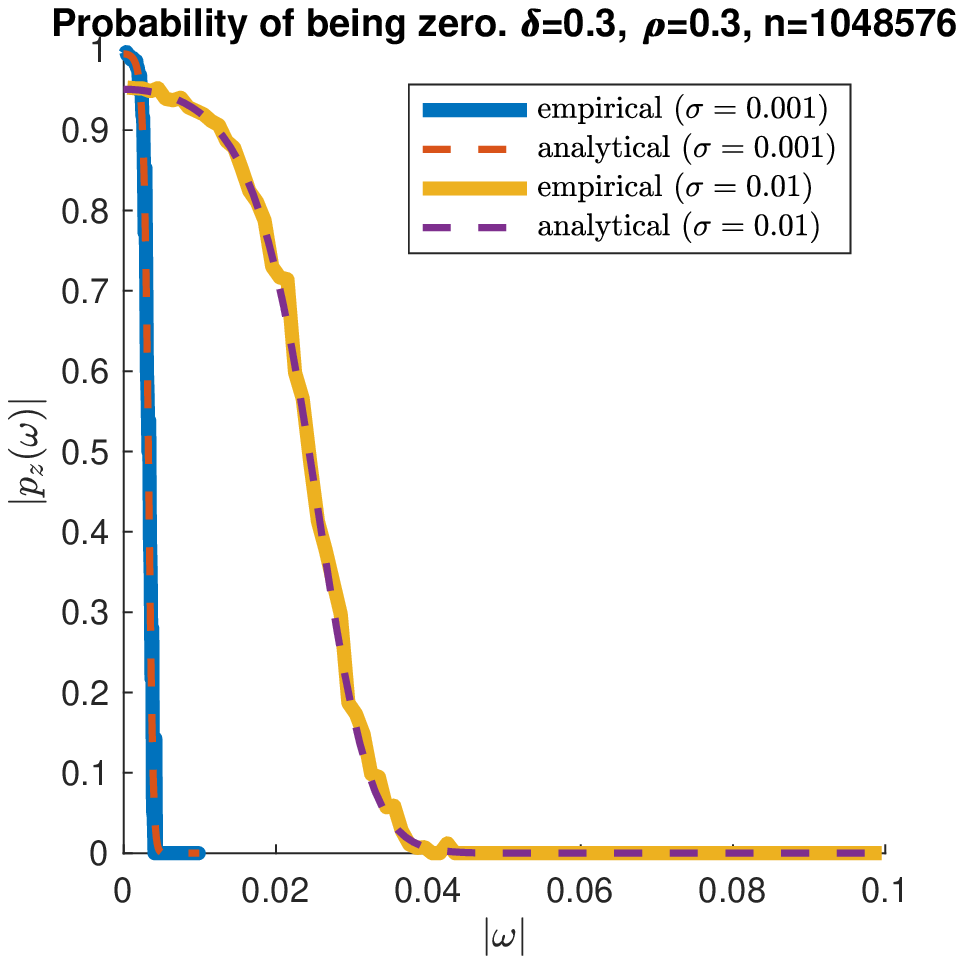}\label{fig:ana-vs-emp-b}}\\
	\subfloat[~]{\includegraphics[width=0.49\linewidth]{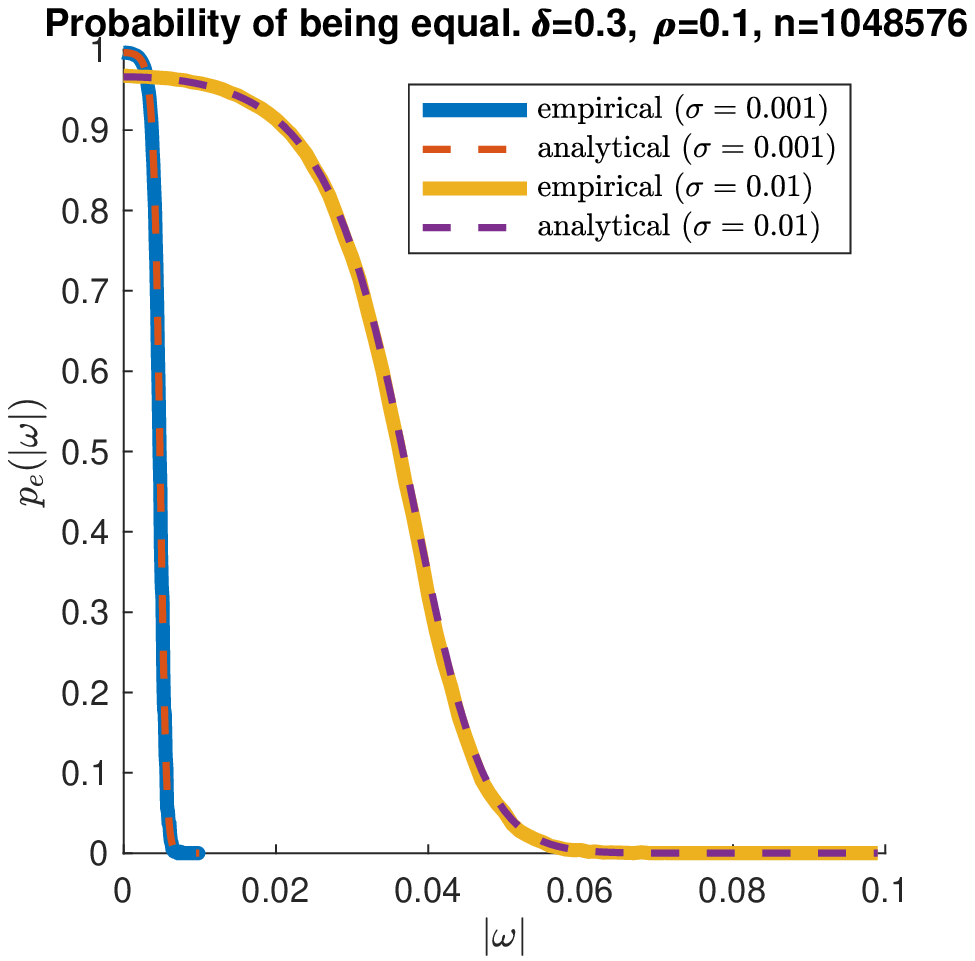}\label{fig:ana-vs-emp-c}}
	\subfloat[~]{\includegraphics[width=0.49\linewidth]{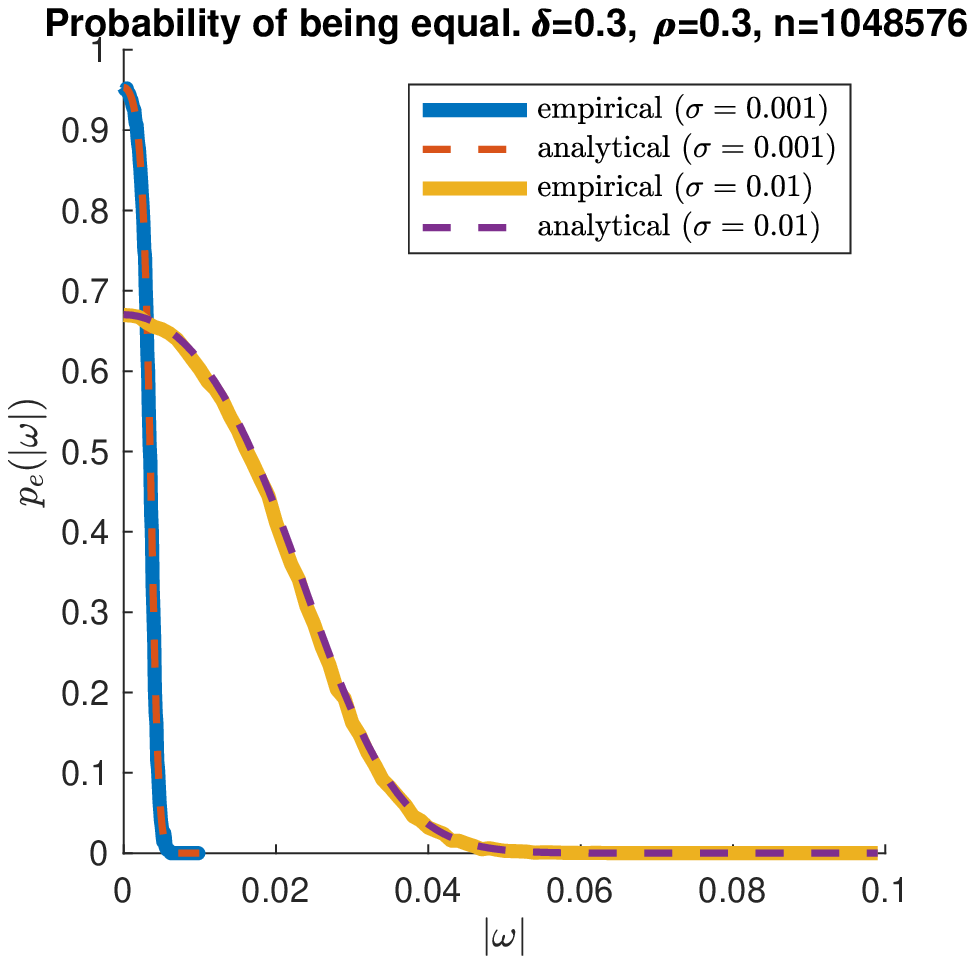}\label{fig:ana-vs-emp-d}}
	\caption{Comparison of analytical and empirical probabilities of a value in the residual being zero or two values in the residual being equal for $\rho\in\set{0.1, 0.3}$ and $\sigma_n\in\set{10^{-3},10^{-2}}$.}
	\label{fig:empirical-vs-analytical-probabilities}
\end{figure}

Overall the analytical expressions fit the empirical probabilities very well, indicating that the approximations we made in the calculations above are justified. 
However, we observe that as $\rho$ and especially $\sigma$ increase, both functions drop significantly. 
This means that for large values of these parameters, the noise eventually dominates and it is difficult to decided whether values are zero or equal.


\subsection{Scaled probabilities $\breve p_z$ and $\breve p_e$}
\label{ssub:scaling-of-probabilities}

We mentioned in Section \ref{sec:introduction} that our algorithms do not implement the functions $p_z$ and $p_e$ exactly, but a scaled version of these.
As observed in Figure \ref{fig:empirical-vs-analytical-probabilities}, the value of $\max_s p_e(s)$ and $\max_s p_z(s)$ varies significantly as $\sigma$ and $\rho$ change.
Algorithm \ref{alg:robust-l0} evaluates whether a given score is large or not by implementing a sweeping parameter $t$ that is set to one at the beginning of the algorithm and decreased by a constant $c$ after every iteration.
In order to use a fixed initial $t$ we consider the scaled probabilities $\breve p_e$ and $\breve p_z$ in $\eqref{eq:breve}$; otherwise the initial value of $t$ would depend on $\sigma$ and $\rho$.
\subsection{Adaptive $k$}

Algorithm \ref{alg:robust-l0} can optionally account for the sparsity of the current estimate via the flag {\tt adaptive\_k}.
In the noiseless model of $r = A\hat x$, an update of the form $\hat x \leftarrow \hat x + \omega e_j$ with Parallel-$\ell_0$ guarantees that $j \in \supp(x)$ so that at the next iteration the problem with residual $r - a_j \omega$ and $(k-1)$-sparse signal is considered.
The {\tt adaptive\_k} flag updates the sparsity prior in $\hat x$ after every update in hope of having more reliable estimates of $p_e$ and $p_z$.
We will see in the numerical experiments that under the data generating model that we tested this strategy does not bring substantial benefits.
We don't rule out the posssibility that there are other signal and noise distributions for which this flag becomes especially useful, but we leave that as future work.

\section{Numerical experiments}
\label{sec:numerical-evidence}

In this section we present numerical experiments which validate the
efficacy of Robust-$\ell_0$ decoding.  In particular, we contrast
Robust-$\ell_0$ with other state-of-the-art greedy algorithms for
compressed sensing in terms of their ability to recover the measured
signal for varying problem sizes $(k,m,n)$ as well as their
computational complexity.  To facilitate reproducibility we begin by describing the stopping
conditions and measures used to denote successful recovery in the
presence of noise in Section \ref{subsec:stop}, along with how the
parameter $c$ is varied in Section \ref{subsec:c}.  We then present
the main numerical results in Section \ref{subsec:phase} where the
algorithms phase transitions and runtime are presented, along with
Sections \ref{subsec:sigma} and \ref{subsec:delta} which show further
details on Robust-$\ell_0$ decoding's performance as a function of
noise variance and for extreme subsampling respectively.

\begin{figure*}[!htbp]
	\centering
  \subfloat[Phase transition $\sigma = 0.001$]{\includegraphics[width=0.24\linewidth]{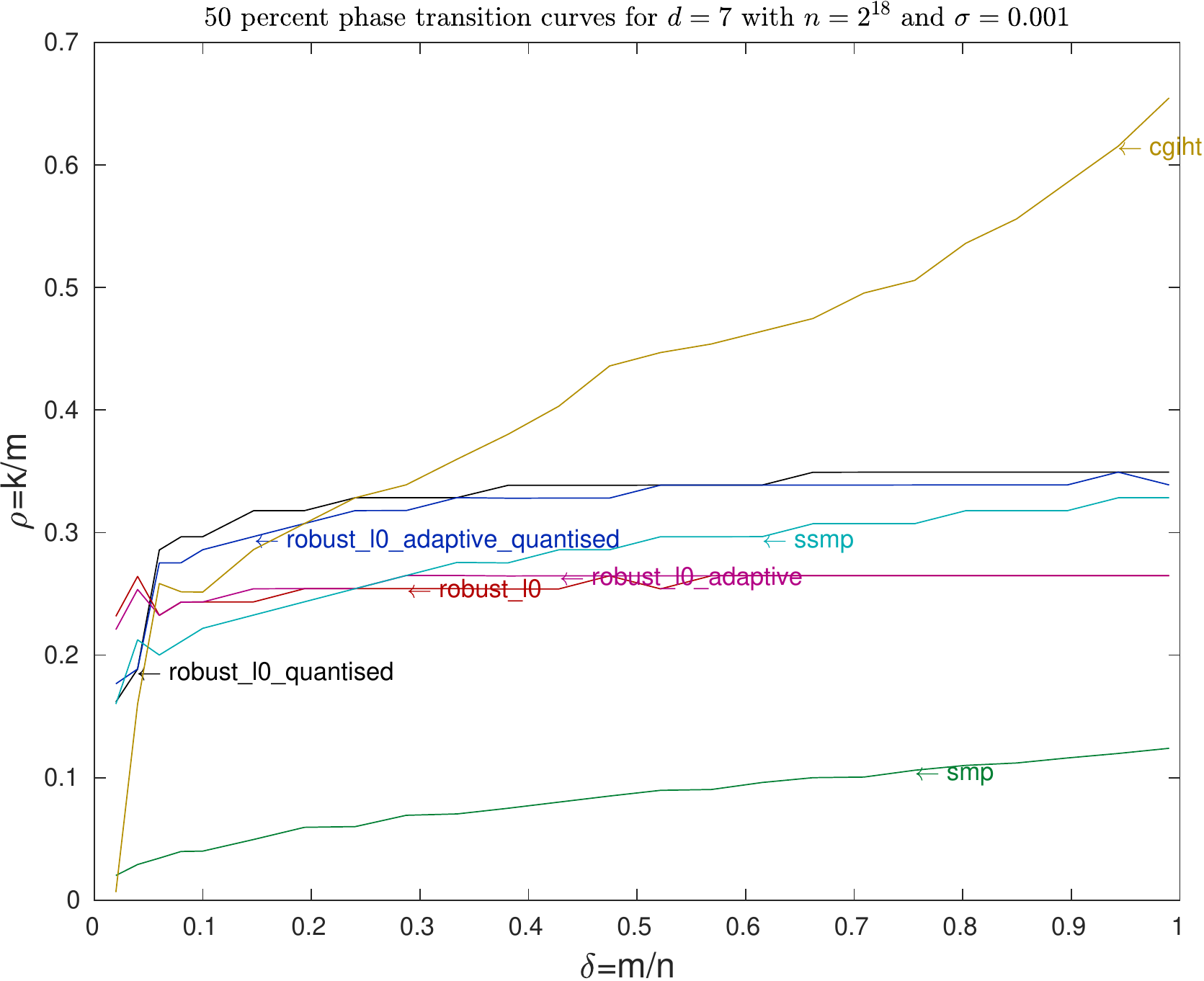} \label{fig:transition_001}}
  \subfloat[Selection map $\sigma = 0.001$]{\includegraphics[width=0.24\linewidth]{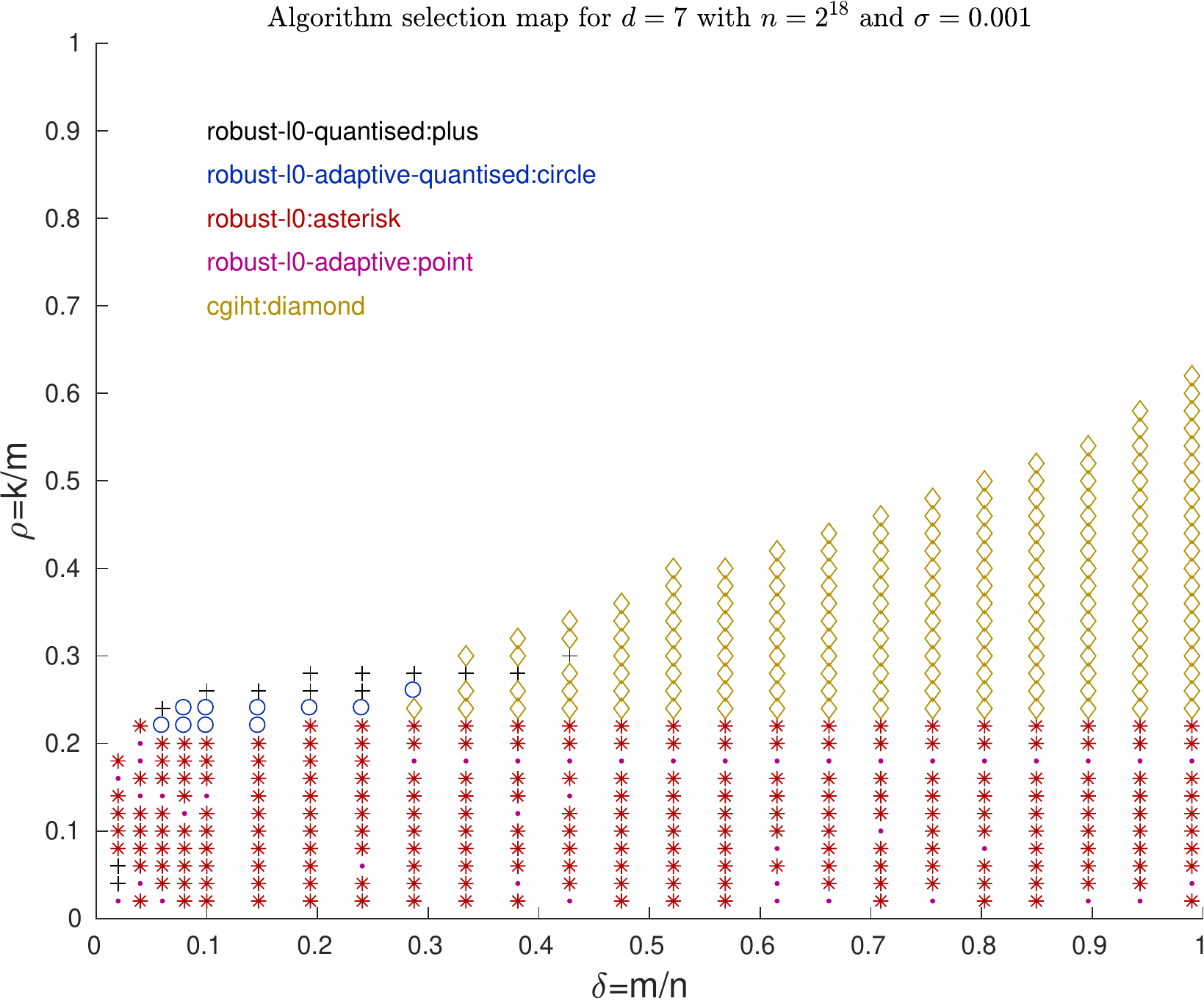} \label{fig:selection_001}}
  \subfloat[Best time $\sigma = 0.001$]{\includegraphics[width=0.24\linewidth]{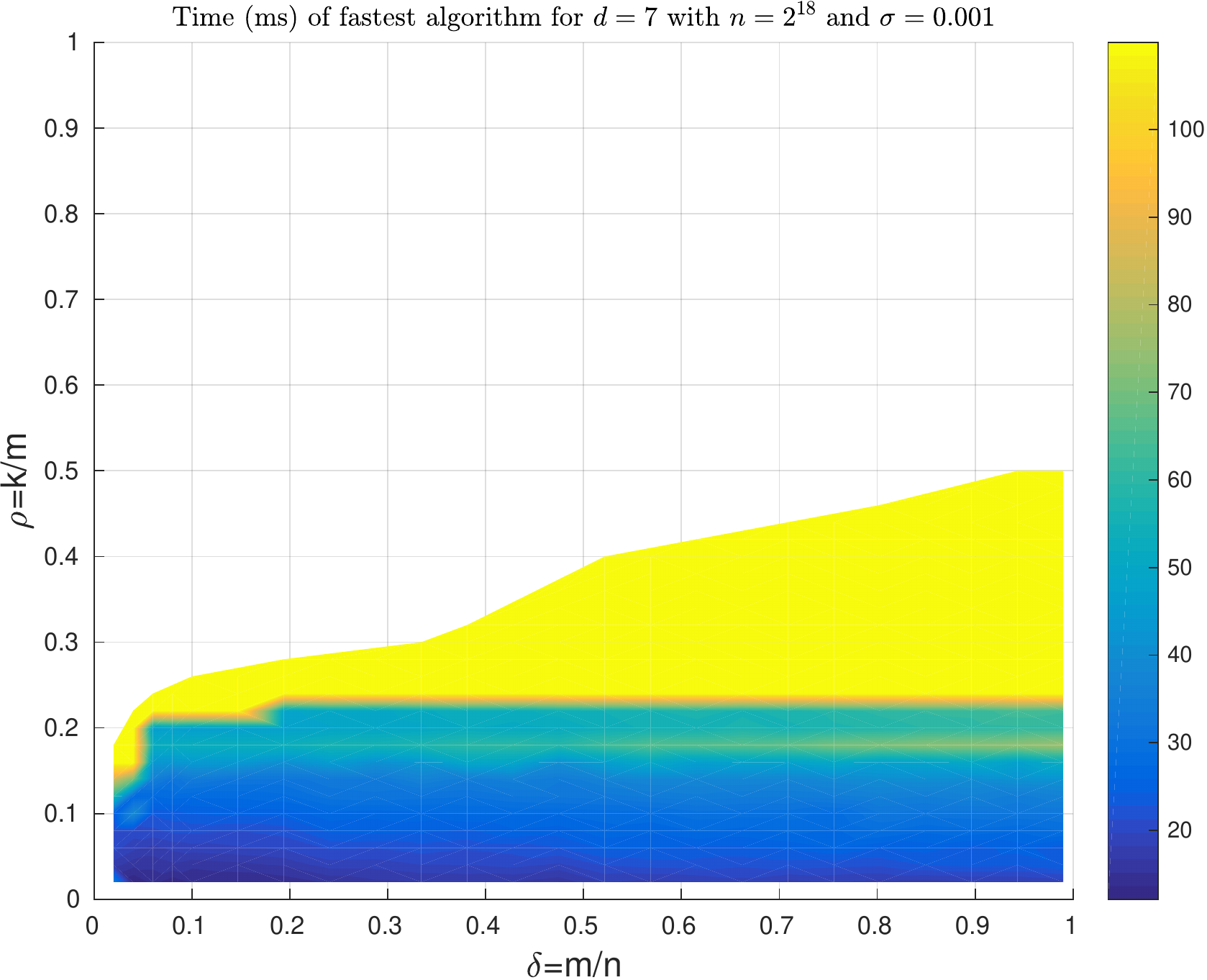} \label{fig:timing_001}}
  \subfloat[Timing ratio: Robust-$\ell_0$, $\sigma = 0.001$]{\includegraphics[width=0.24\linewidth]{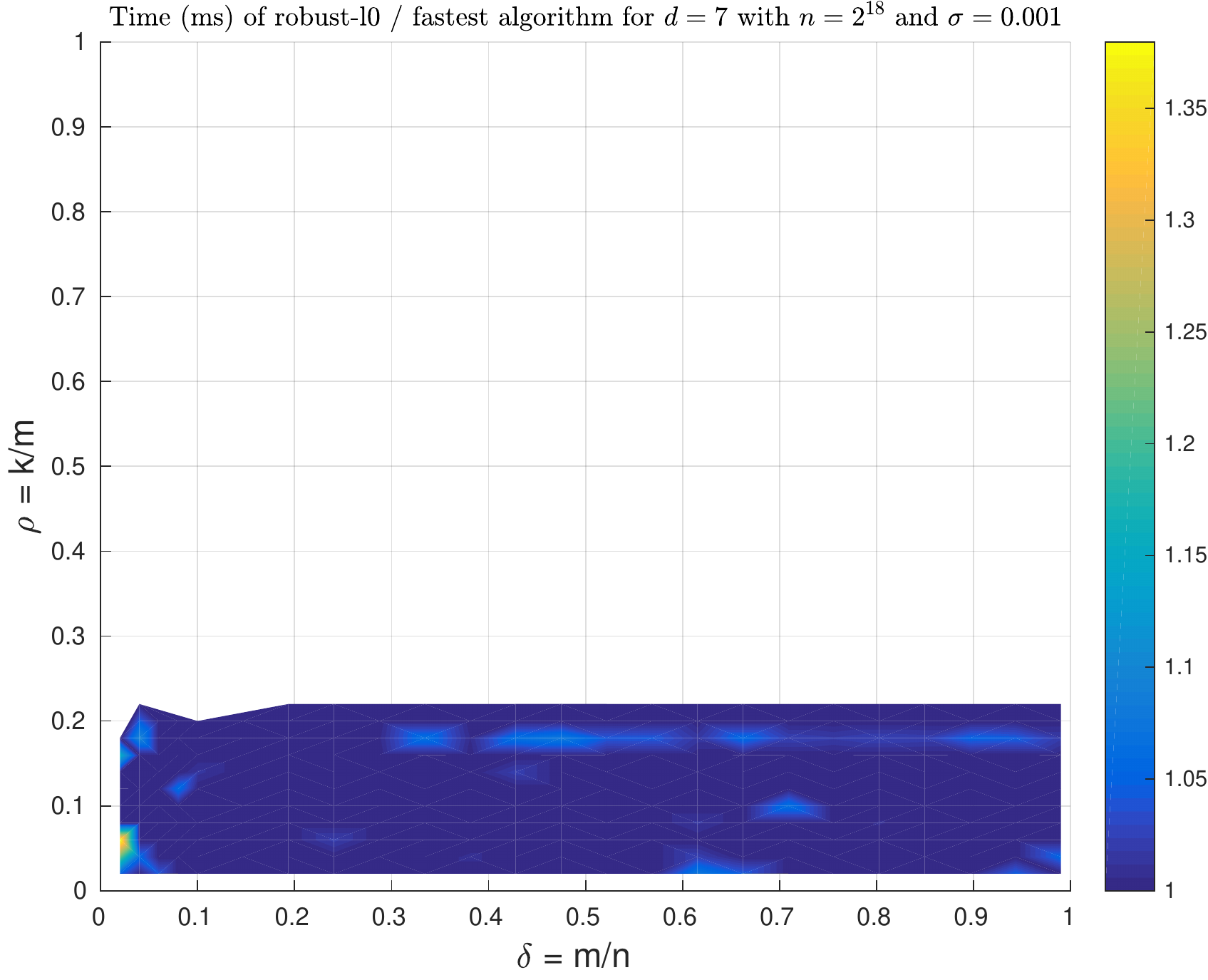} \label{fig:ratio_001}}\\

  \subfloat[Phase transition $\sigma = 0.01$]{\includegraphics[width=0.24\linewidth]{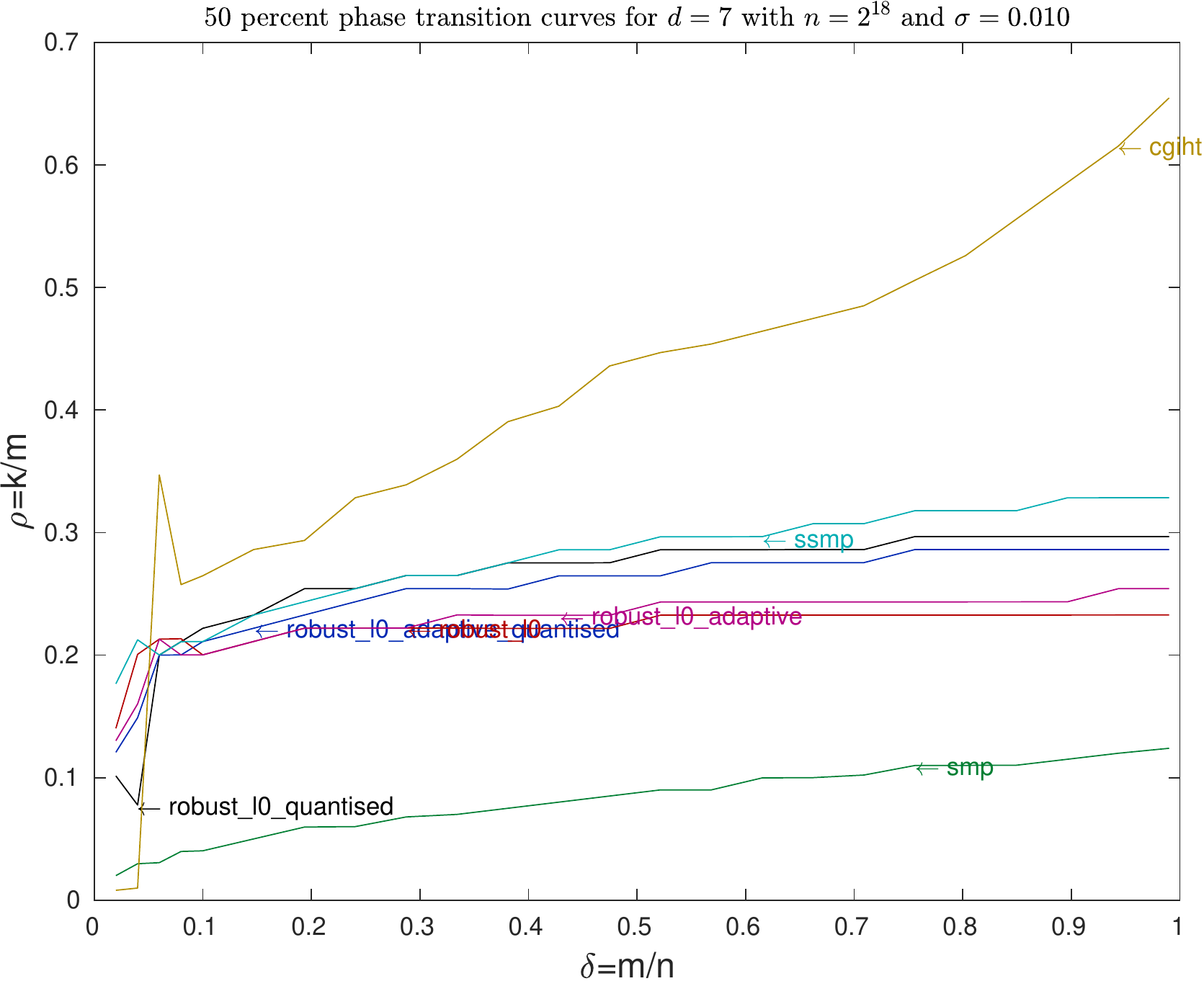} \label{fig:transition_010}}
  \subfloat[Selection map $\sigma = 0.01$]{\includegraphics[width=0.24\linewidth]{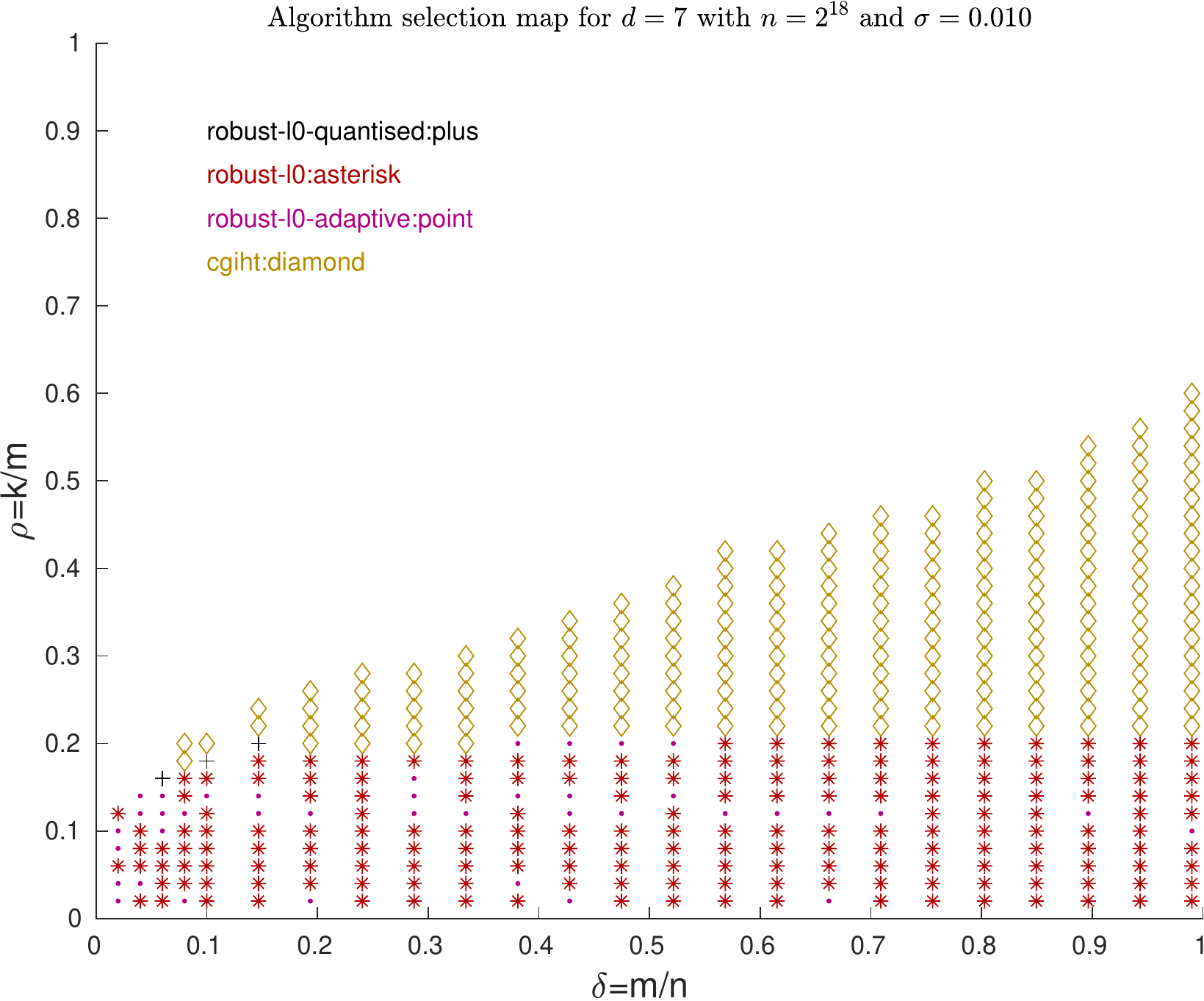} \label{fig:selection_010}}
  \subfloat[Best time $\sigma = 0.01$]{\includegraphics[width=0.24\linewidth]{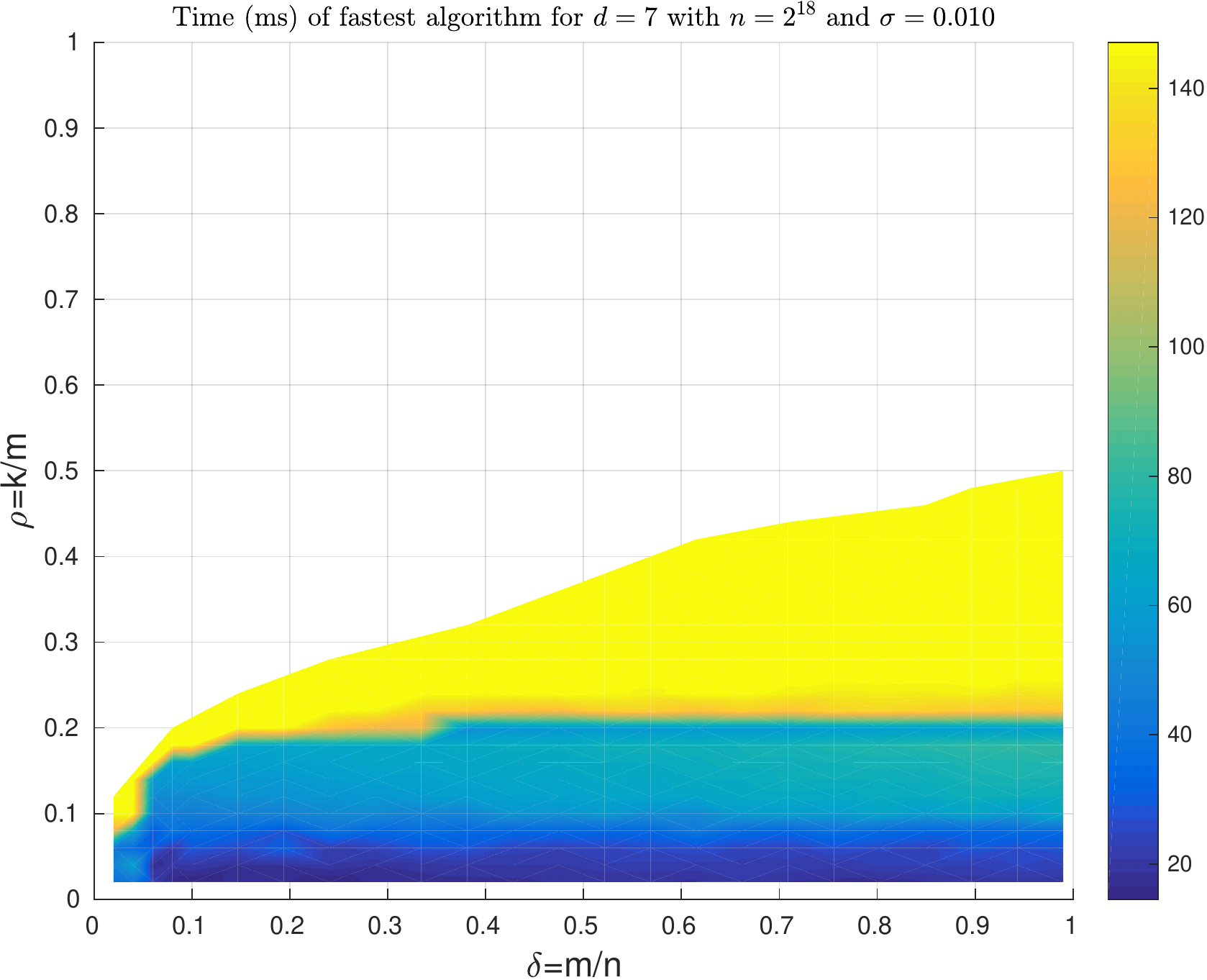} \label{fig:timing_010}}
  \subfloat[Timing ratio: Robust-$\ell_0$, $\sigma = 0.001$]{\includegraphics[width=0.24\linewidth]{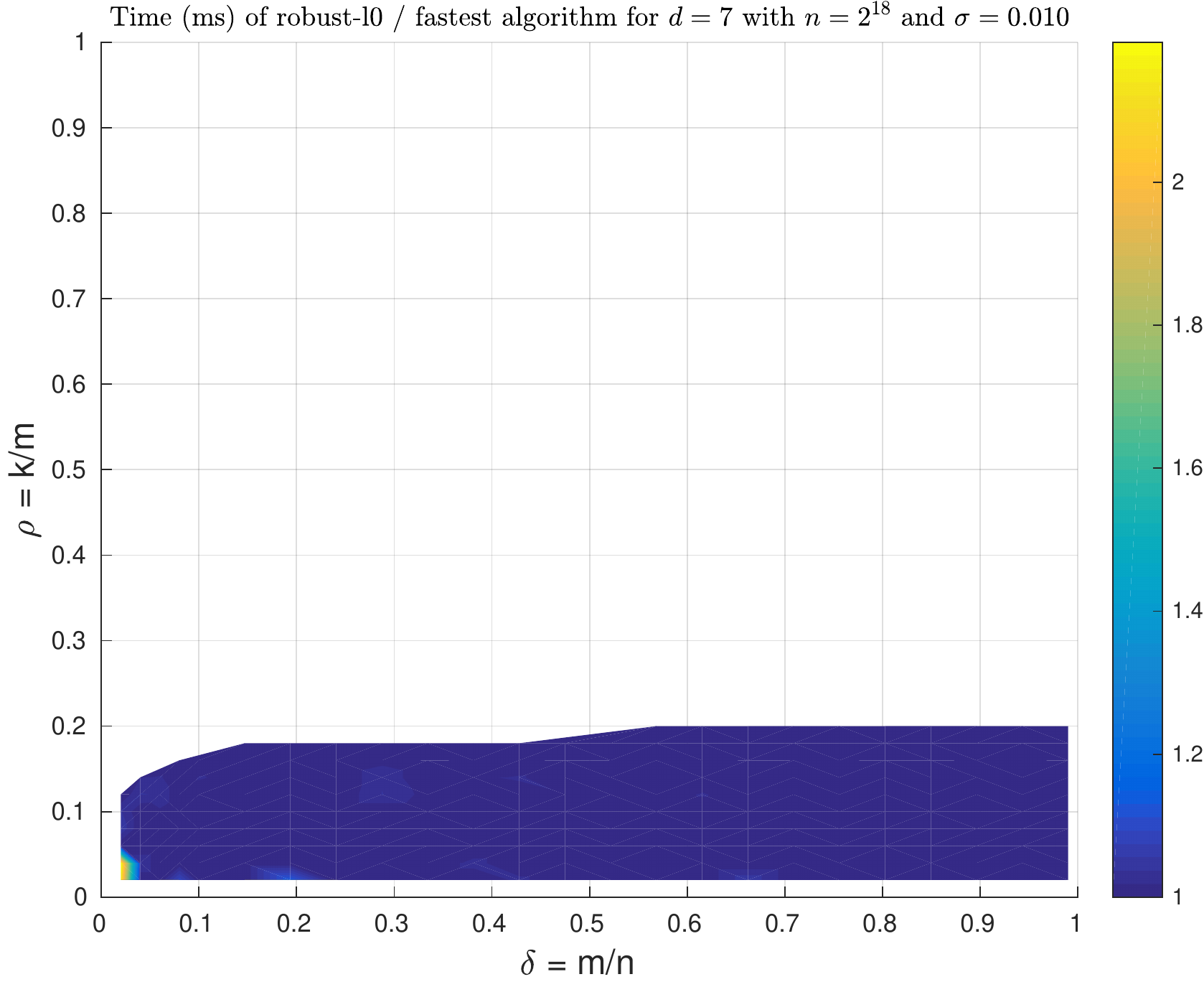} \label{fig:ratio_010}}\\

  \subfloat[Phase transition $\sigma = 0.1$]{\includegraphics[width=0.24\linewidth]{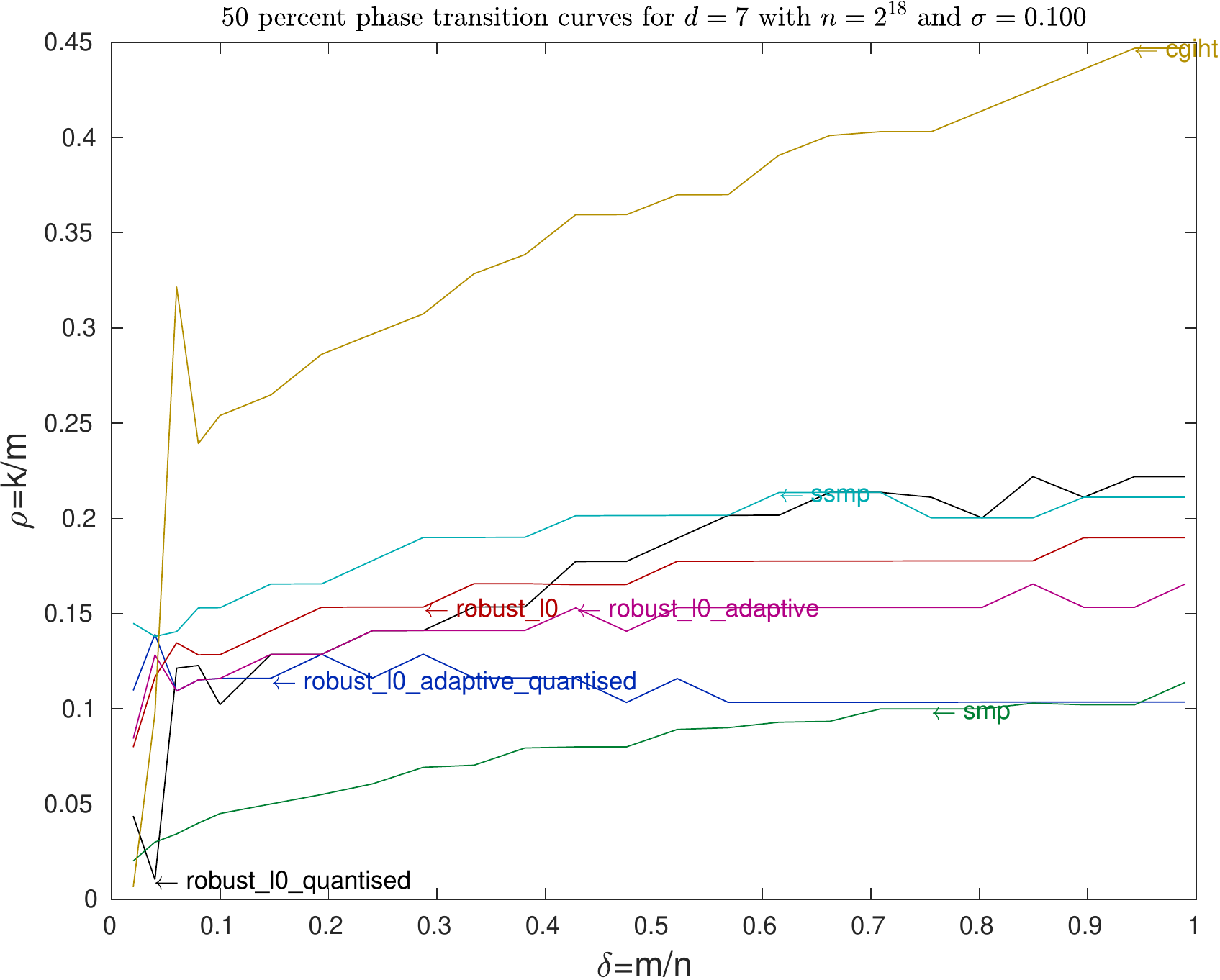} \label{fig:transition_100}}
  \subfloat[Selection map $\sigma = 0.1$]{\includegraphics[width=0.24\linewidth]{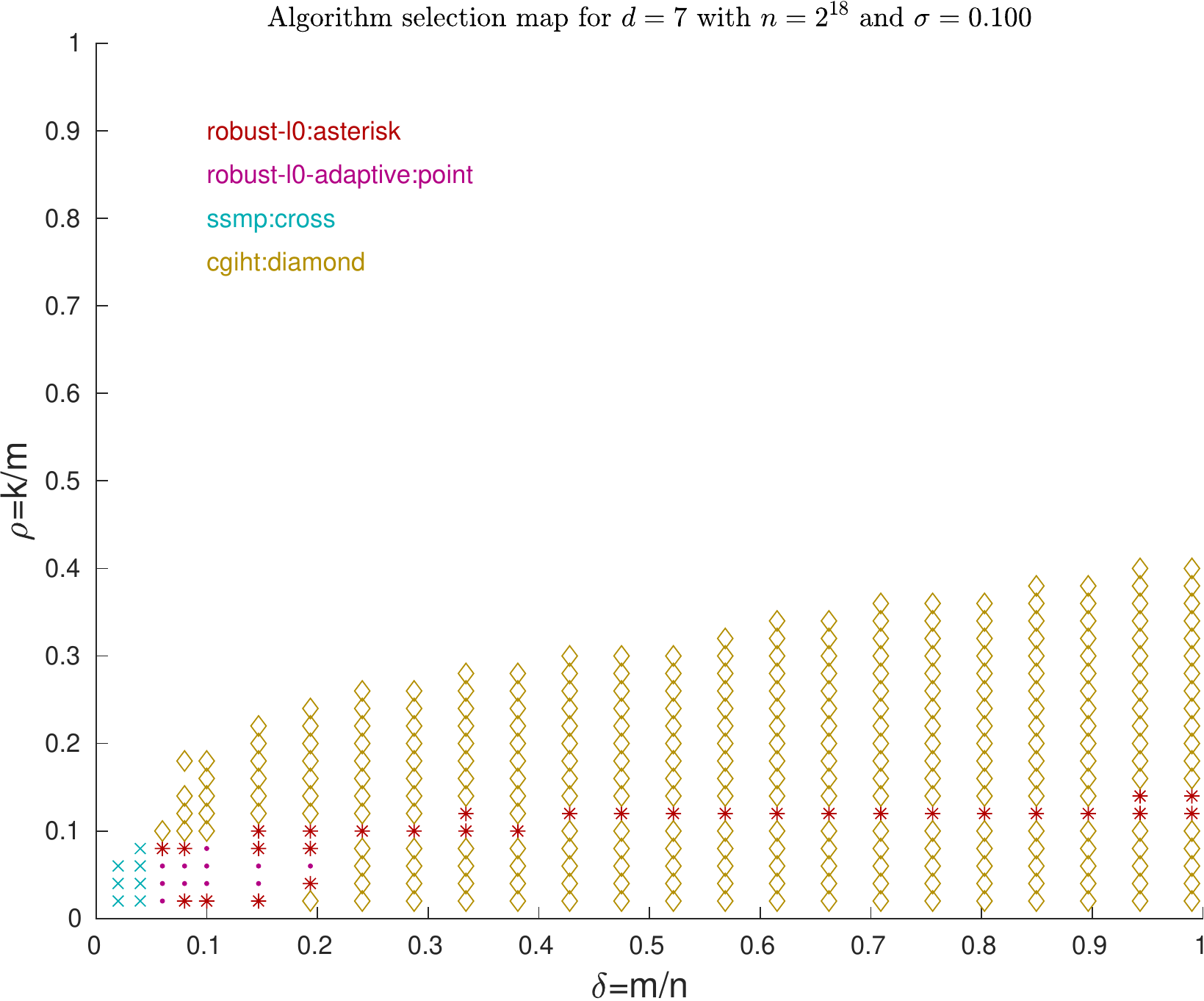} \label{fig:selection_100}}
  \subfloat[Best time $\sigma = 0.1$]{\includegraphics[width=0.24\linewidth]{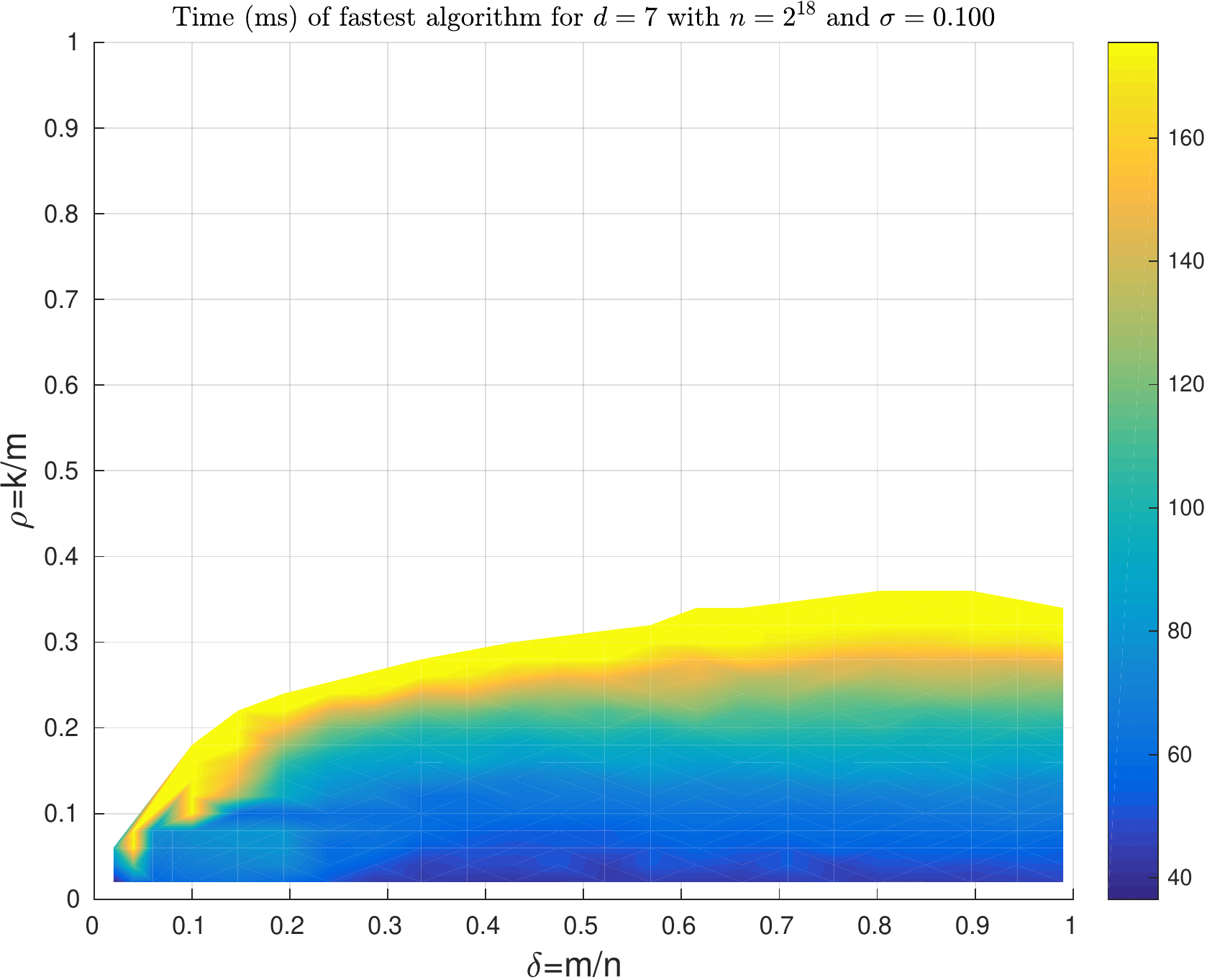} \label{fig:timing_100}}
  \subfloat[Timing ratio: Robust-$\ell_0$, $\sigma = 0.001$]{\includegraphics[width=0.24\linewidth]{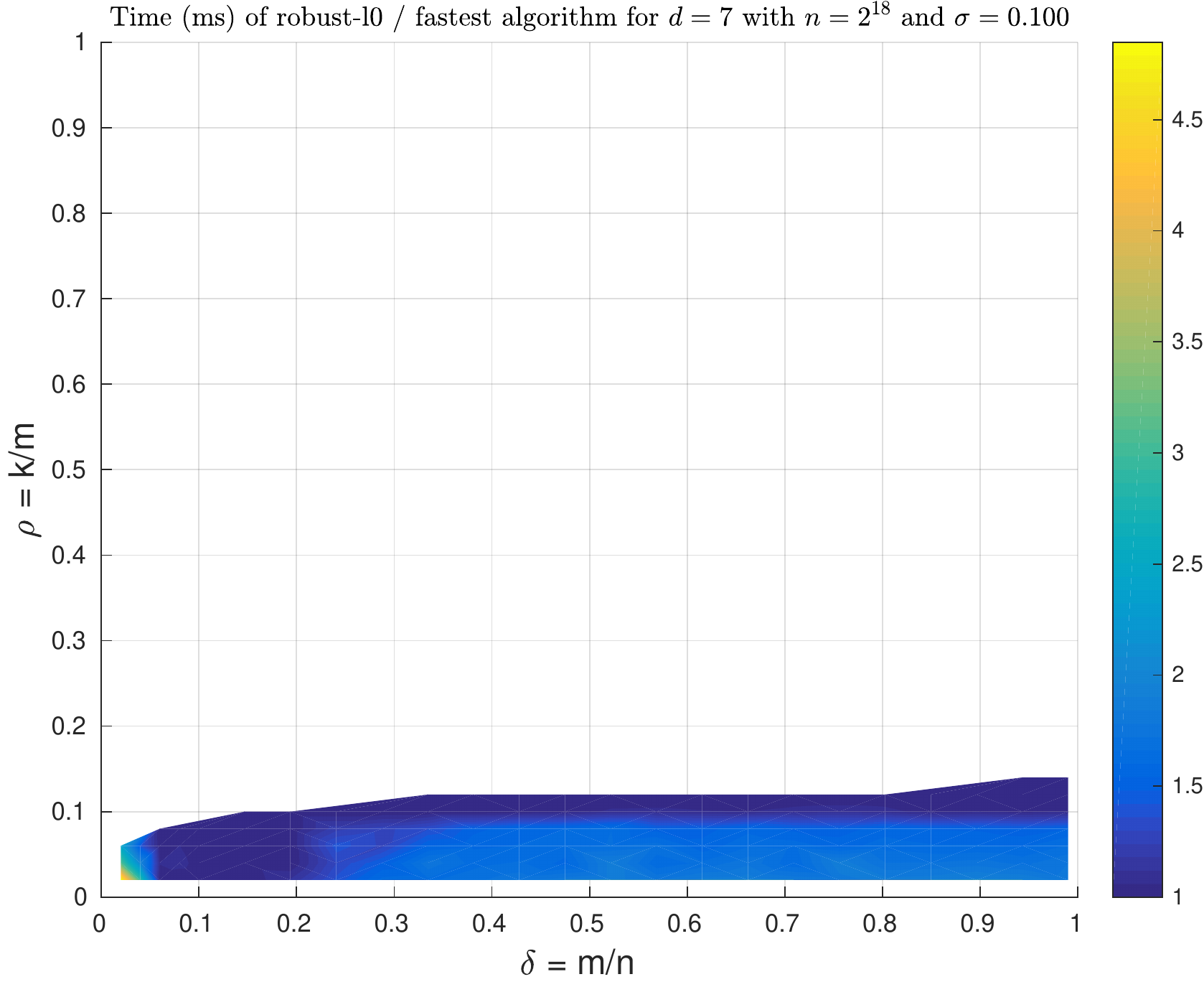} \label{fig:ratio_100}}
\caption{Phase transitions, selection maps and timings for $n = 2^{18}$.}
	\label{fig:selection_maps_best_time_n_2p18}
\end{figure*}

\subsection{Stopping conditions}\label{subsec:stop}

We are interested in the signal model $y = Ax + \eta$ where $\eta \sim \mathcal{N}(0, \sigma^2 I_{m \times m})$.
If $\hat x$ is an approximation to $x$ the residual is $r = y - A\hat x$.
Note that if $\hat x = x$, then
\begin{align*}
\|r\|_1 &= \|y - A\hat x\|_1 \\
& = \|y - Ax\|_1\\
& = \|\eta\|_1
\end{align*}
and we should not seek reductions in the residual below $\|\eta\|_1$
which would result in fitting to the additive noise.
We further account for the variance of $\|\eta\|_1$ and denote the algorithm to have successfully recovered $x$ if $\hat{x}$ satisfies
\begin{equation}
\label{eq:success_condition}
\frac{\|x - \hat x\|_1}{\|x\|_1} \leq \frac{\mathbb{E}\left[\|\eta\|_1\right] + C_1 \sqrt{\Var\left[\|\eta\|_1\right]}}{\|x\|_1}
\end{equation}
for some $C_1 \geq 0$.
We should be aware that the right hand side of \eqref{eq:success_condition} might be greater than 1 for some choices of $k$, $m$ and $\sigma$.
When this happens, the stopping condition \eqref{eq:success_condition} becomes invalid since we expect $\hat x$ to have captured a proportion of the $\ell_1$-energy of $\|x\|_1$.
Hence, if the right hand side of \eqref{eq:success_condition} is greater than $\frac{1}{10}$ we clip the upper bound at this value and use the stopping condition
\begin{equation}
\label{eq:success_condition_clip}
\frac{\|x - \hat x\|_1}{\|x\|_1} \leq \min\left(\frac{\mathbb{E}\left[\|\eta\|_1\right] + C_1 \sqrt{\Var\left[\|\eta\|_1\right]}}{\|x\|_1}, \frac{1}{10}\right).
\end{equation}
For the numerical experiments conducted in this section we consider nonzero entries in $x$ drawn as $x_i\sim \mathcal{N}(0,1)$, and noise $\eta_i \sim \mathcal{N}(0, \sigma^2)$ for which $\mathbb{E}\left[\|x\|_1 \right] = k\sqrt{\frac{2}{\pi}}$ and $\mathbb{E}\left[\|\eta\|_1 \right] = m\sigma \sqrt{\frac{2}{\pi}}$ and moreover $\Var\left[\|\eta\|_1 \right] = m\sigma^2\left(1 - \frac{2}{\pi}\right)$, see e.g. \cite{leone1961folded}.
%
%

%

\subsection{Selection of parameter $c$ in Algorithm \ref{alg:robust-l0}}\label{subsec:c}

The sweeping parameter $t$ in Algorithm \ref{alg:robust-l0} is initialised at 1 and updated by decreasing it by a constant value $c$.
We observed in our experiments that the quality of the phase transitions are sensitive on the parameter $c$ especially for low $\delta$ and $\rho$.
We do not provide a way to fine-tune $c$, but we run our phase transitions with the following choices:
\begin{enumerate}
\item If the algorithm is {\tt quantised},
\begin{equation}
\label{eq:finetune_c_quant}
c = \left\{\begin{array}{ll}
0.01 & \delta \leq 0.05\\
0.05 & \delta > 0.05  \mbox{ and } \rho \in (0, 0.1]\\
0.075 & \delta > 0.05  \mbox{ and } \rho \in (0.1, 0.2]\\
0.1 & \delta > 0.05 \mbox{ and } \rho \in (0.2, 1)
\end{array}\right..
\end{equation}
\item If the algorithm is {\tt continuous},
\begin{equation}
\label{eq:finetune_c_cont}
c = \left\{\begin{array}{ll}
0.01 & \delta \leq 0.05\\
0.025 & \delta > 0.05
\end{array}\right..
\end{equation}
\end{enumerate}
The values in \eqref{eq:finetune_c_quant} and
\eqref{eq:finetune_c_cont} were chosen heuristically for $\nu$ and
$\mu$ Gaussian.

\subsection{Phase transitions and runtime}\label{subsec:phase}
We benchmark the variants of Robust-$\ell_0$ against other greedy
algorithms via their {\em phase-transitions and runtime}.
The user can supply two binary flags, {\tt adaptive\_k} and {\tt
  quantised} which yield four different variants of Robust-$\ell_0$.
We assigned a unique label to each of these variants as described in Table \ref{table:variants}.

\begin{table}
\begin{tabular}{ l || c | c }
Algorithm & {\tt \footnotesize adaptive\_k}& {\tt \footnotesize quantised}\\
\hline
\hline
  Robust-$\ell_0$                 & No & No \\
  Robust-$\ell_0$-adaptive        & Yes & No \\
  Robust-$\ell_0$-quantised           & No & Yes \\
  Robust-$\ell_0$-adaptive-quantised  & Yes & Yes \\
\end{tabular}
\caption[Variants of Robust-$\ell_0$]{Variants of Robust-$\ell_0$}
\label{table:variants}
\end{table}
The phase transition of a compressed-sensing algorithm \cite{donoho2010precise} is the largest value of $k/m$ for which the algorithm is typically able recovery all $k$ sparse vectors with sparsity less than $k$ for a fixed $m/n$.
Hence, for a fixed value of $\delta = m/n$ the phase transition of an algorithm is the largest value $\rho^*(\delta)$ for which the algorithm converges for all $\rho(\delta) < \rho^*(\delta)$.
The value $\rho^* (m/n)$ often converges to a fixed value as $n \rightarrow \infty$, so phase transitions often partition the $\delta\times \rho$ space into two regions: One in which the algorithm converges with high probability and another in which the algorithm doesn't converge with high probability.
We benchmark Robust-$\ell_0$ against the algorithms presented in \cite{Berinde:2008aa, Berinde:2009aa, Blanchard:2015aa}. Specifically, our tests include the following algorithms,
\begin{center}
\{{\small Robust-$\ell_0$, Robust-$\ell_0$-adaptive, Robust-$\ell_0$-quantised, Robust-$\ell_0$-adaptive-quantised, SSMP, SMP, CGIHT}\}.
\end{center}
In the deterministic case Parallel-$\ell_0$ was compared against a range of combinatorial compressed sensing algorithms in \cite{Tanner:2015aa}; out of those we have selected SSMP and SMP as these perform best and are similar in nature to Robust-$\ell_0$.
We also compare with CGIHT, as this algorithm was shown to be the fastest among the greedy algorithms compared in \cite{Blanchard:2015aaCGIHT, Blanchard:2015aa}.
%
Figures \ref{fig:transition_001}, \ref{fig:transition_010} and
\ref{fig:transition_100} show the phase transition curves for these
algorithms with $\sigma=10^{-3}, \; 10^{-2},\; 10^{-1}$ respectively.
The curves were computed by setting $n = 2^{18}$, $d= 7$, and using
the stopping condition $\|r\|_1\le\mathbb{E}\left[\|\eta\|_1 \right] = m\sigma \sqrt{\frac{2}{\pi}}$
and a success condition \eqref{eq:success_condition} with $C_1 =
1$. 
The testing is done at $m=\delta_p n$ for 
\begin{equation*}
\delta_p \in \{ 0.02p :p \in [4]\} \cup \left\{0.1 + \frac{89}{1900}(p - 1) : p \in [20]\right\}.
\end{equation*}
For each $\delta_p$, we set $\rho = 0.01$ and generate 10 synthetic problems to apply the algorithms to, with a problem generated as in $\expandermodel$ with the given parameters and $\mu$ and $\nu$ being normal Gaussian.
If at least one such problem was recovered successfully, the sparsity ratio $\rho$ is increased by $0.01$ and the experiment is repeated.
Following the testing framework in \cite{blanchard2013performance},
the recovery data is fitted using a logistic function and finally the
50\% recovery transition function is computed and presented in the
phase transition plots contained herein. 
Figures \ref{fig:selection_001}, \ref{fig:selection_010} and \ref{fig:selection_100} show a selection map for these algorithms.
Namely, these plots indicate which algorithm requires the least
computational time\footnote{All the numerical results presented here were performed using a Linux machine with Intel Xeon E5-2643 CPUs \@ 3.30 GHz, NVIDIA Tesla K10 GPUs and executed from Matlab R2016b.
The code was added to the GAGA library available at
\texttt{\url{http://www.gaga4cs.org/}}, and described in
\cite{blanchard2012gpu}, so as to facilitate large scale benchmarking
against the other algorithms presented here which are also contained
in the aforementioned library.}
 at each point in the $\delta\times \rho$ space where the algorithm converges.
Finally, Figures \ref{fig:timing_001}, \ref{fig:timing_010} and
\ref{fig:timing_100} show the total time for convergence in
milliseconds for the fastest algorithm at each point in the $\delta
\times \rho$ space.  
The ratio of the time for the second fastest algorithm over the time
for the fastest algorithm is given in Figures \ref{fig:ratio_001},
\ref{fig:ratio_010}, and \ref{fig:ratio_100}.

We can see from Figure \ref{fig:selection_maps_best_time_n_2p18} that CGIHT \cite{Blanchard:2015aa} dominates the upper region of the phase transition space, while the Robust-$\ell_0$ algorithms only converge for $\rho \lessapprox 0.3$ which is consistent with the observed phase transitions for Parallel-$\ell_0$ \cite{Tanner:2015aa}.
In terms of speed, Robust-$\ell_0$ seems to be the most competitive
for $\sigma \in \{10^{-3}, 10^{-2}\}$ and $\rho \lessapprox
0.2$. However, for large noise, $\sigma = 10^{-1}$, CGIHT becomes the fastest algorithm of all.
We remark that while CGIHT performs very well in our numerical tests, the current theory developed for it does not hold in the setting considered here, as it requires zero-mean columns in $A$.

Figure \ref{fig:pt_widths} shows the widths for the Robust-$\ell_0$ algorithms.
The widths measure how sharp the phase transition of an algorithm is; namely, how thin the boundary between the region of recovery with high-probability and the region of recovery where combinatorial search is needed.
It has been shown that the widths of a compressed sensing algorithm
tend to zero as $n \rightarrow \infty$ when decoding with linear
programming \cite{Donoho:2010aa}, and we usually expect the same
behaviour for other algorithms \cite{Blanchard:2015aa}.
Figure \ref{fig:pt_widths} show that the widths for the Robust-$\ell_0$ algorithms indeed decrease with $n$ and with $\delta$.
The observed smoothness of the phase transition widths signal also suggest that the stopping conditions of the algorithm are consistent for the problem under consideration.
\begin{figure}[!htbp]
	\centering
	\subfloat[robust-$\ell_0$-adaptive]{\includegraphics[width=0.49\linewidth]{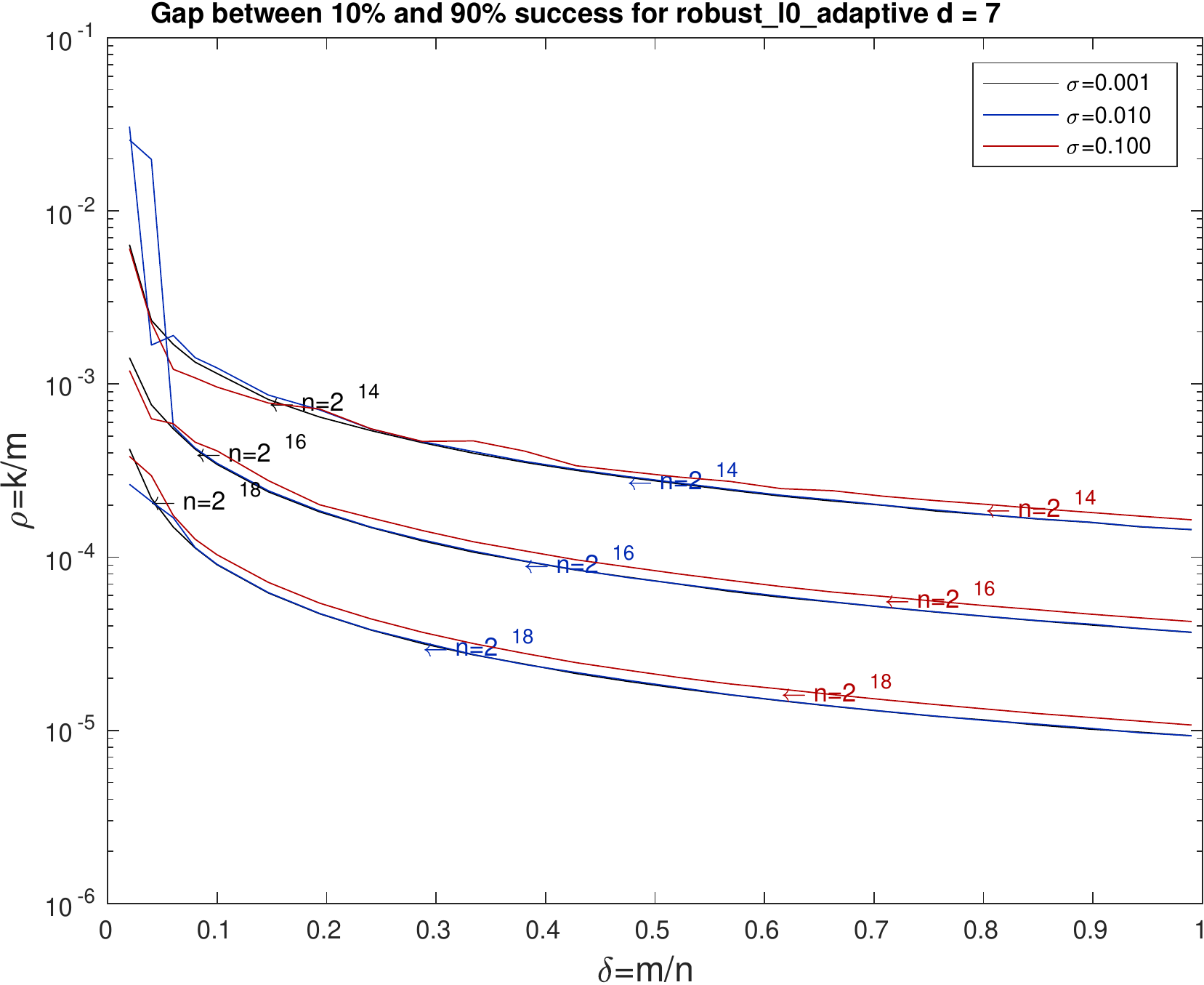}}
	\subfloat[robust-$\ell_0$-adaptive-quantised]{\includegraphics[width=0.49\linewidth]{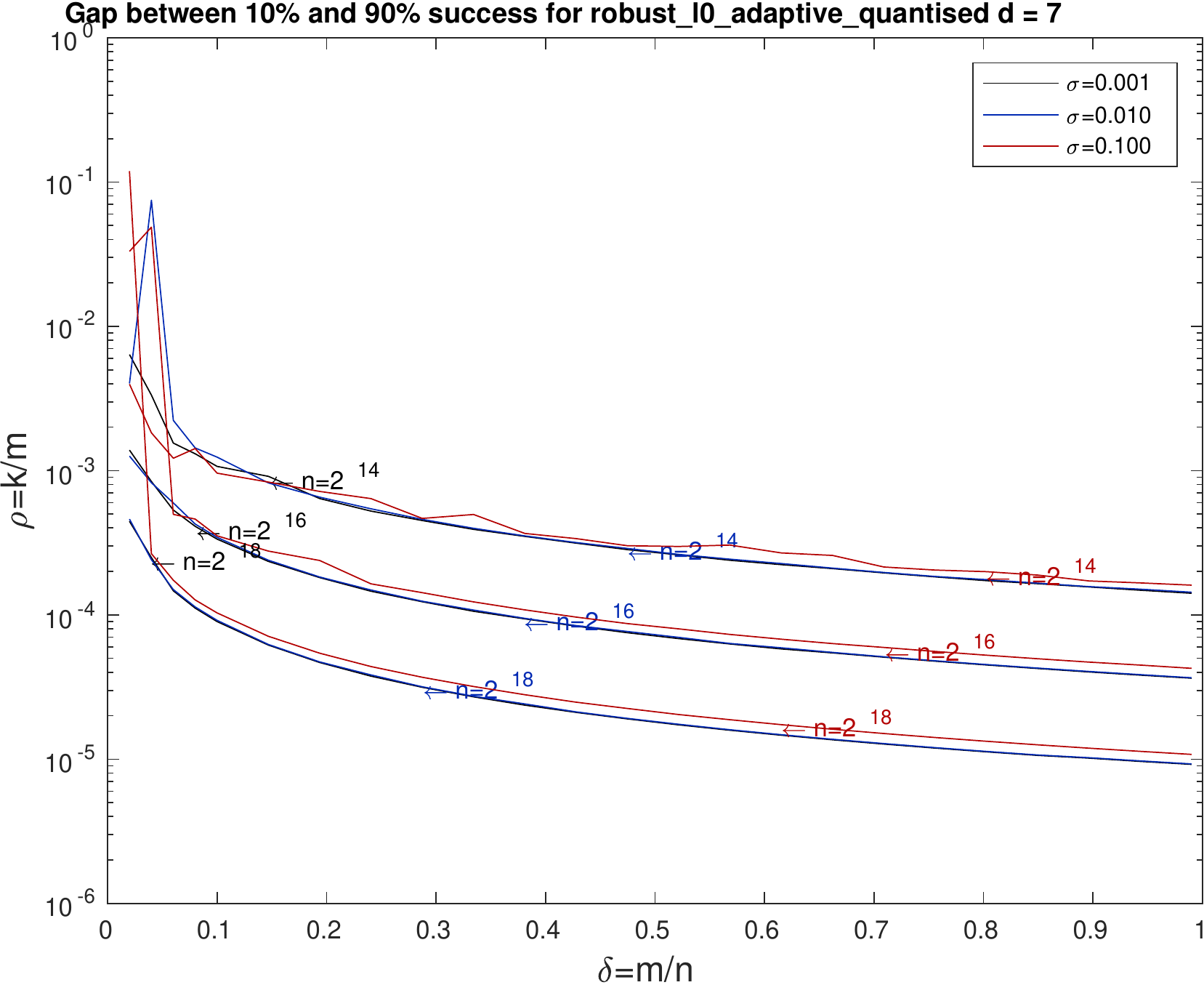}}\\
	\subfloat[robust-$\ell_0$]{\includegraphics[width=0.49\linewidth]{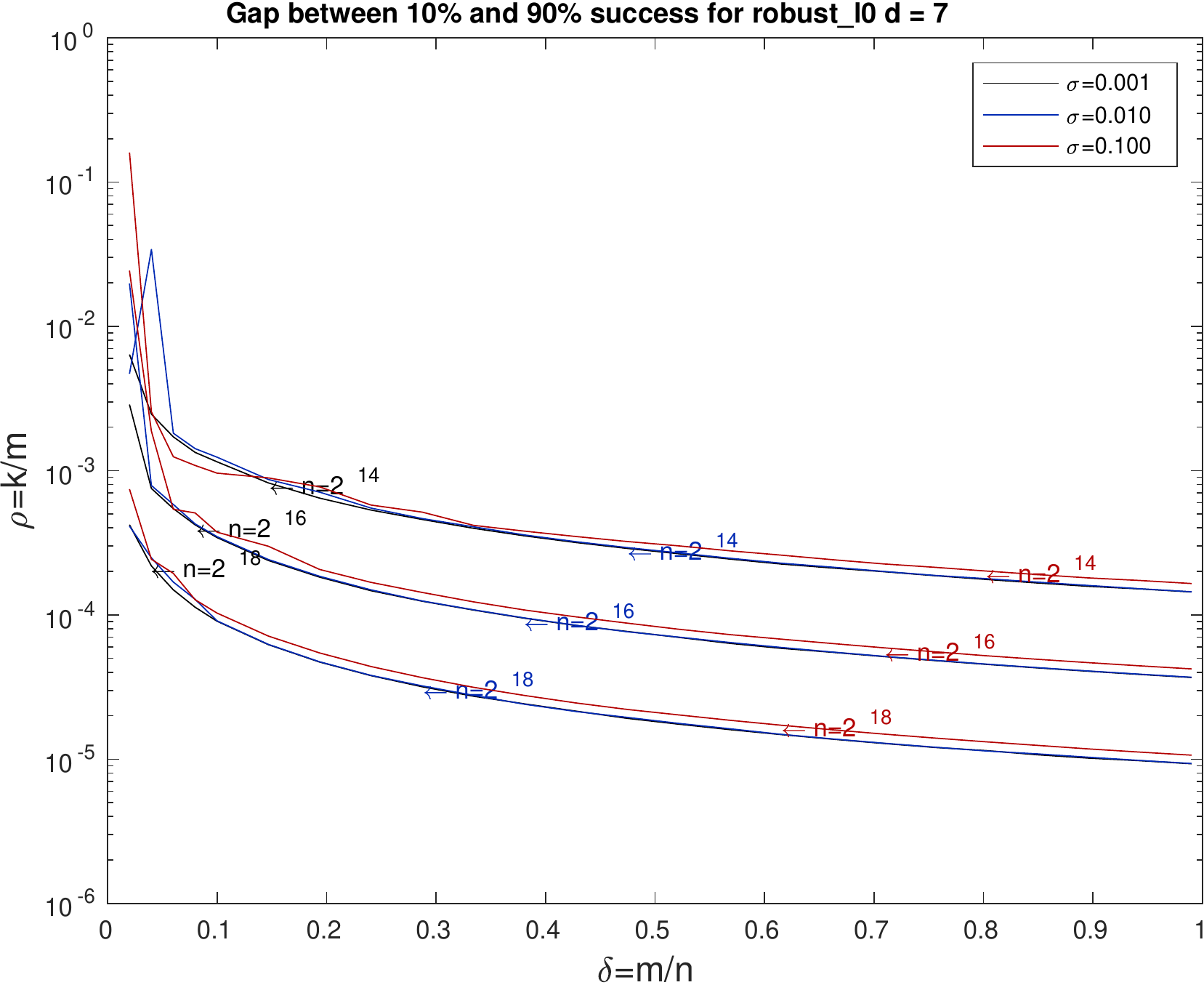}}
	\subfloat[robust-$\ell_0$-quantised]{\includegraphics[width=0.49\linewidth]{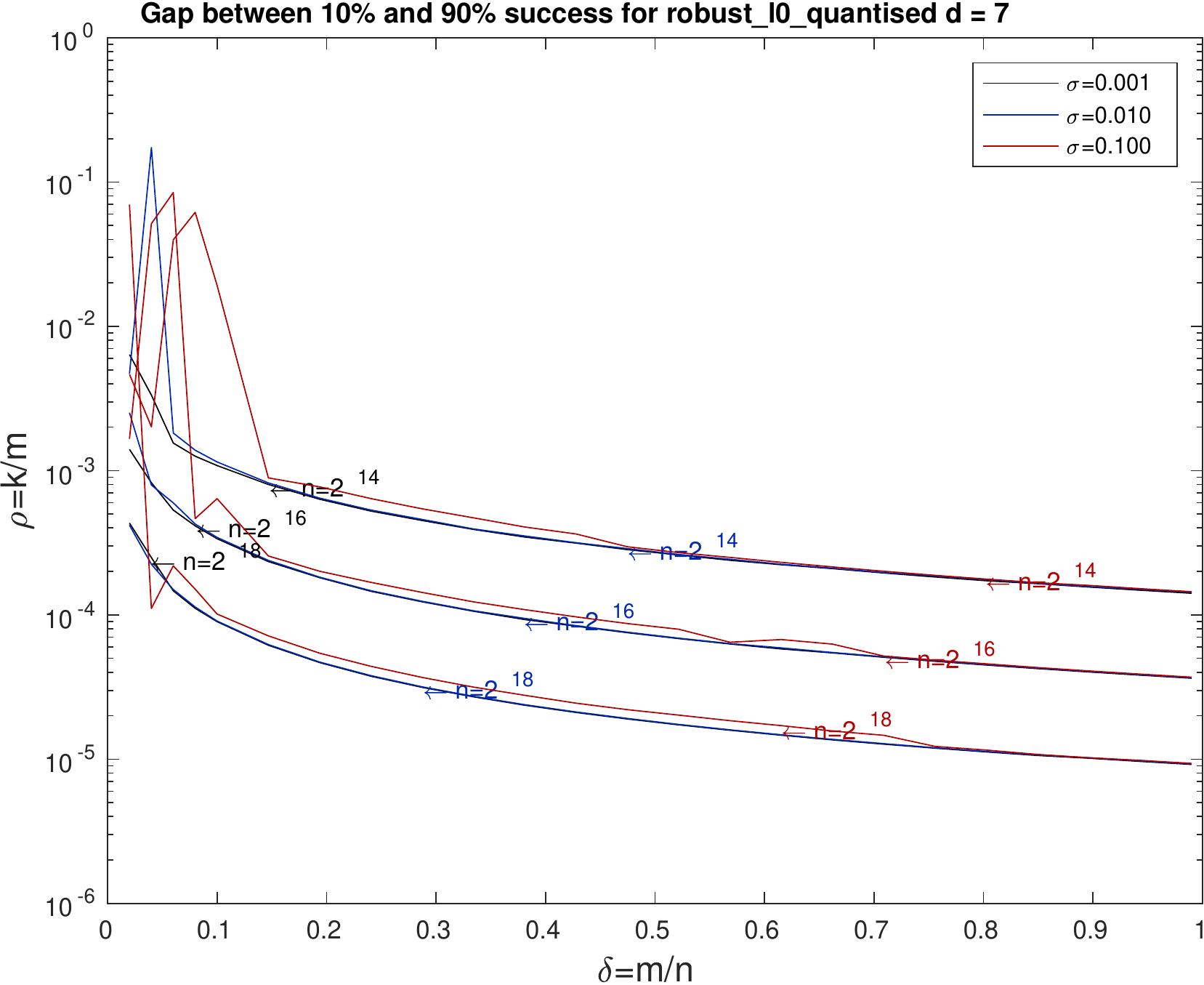}}
	\caption[Phase transition widths for Robust-$\ell_0$ variants.]{Widths for Robust-$\ell_0$ variants}
	\label{fig:pt_widths}
\end{figure}

\subsection{Dependence on noise variance, $\sigma$}\label{subsec:sigma}
\begin{figure}[!htbp]
	\centering
	\subfloat[$\delta=0.01$]{\includegraphics[width=0.49\linewidth]{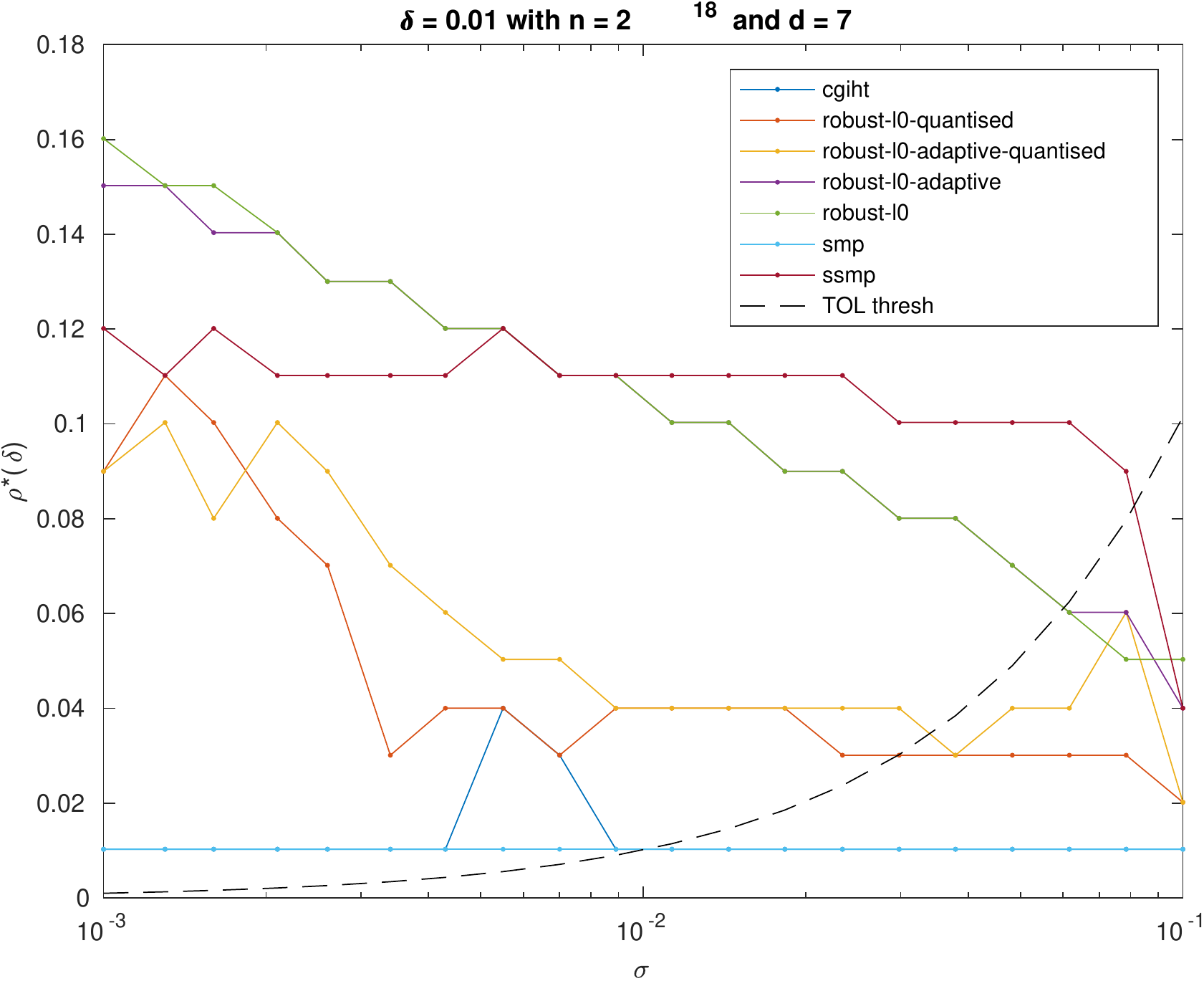}}
	\subfloat[$\delta=0.02$]{\includegraphics[width=0.49\linewidth]{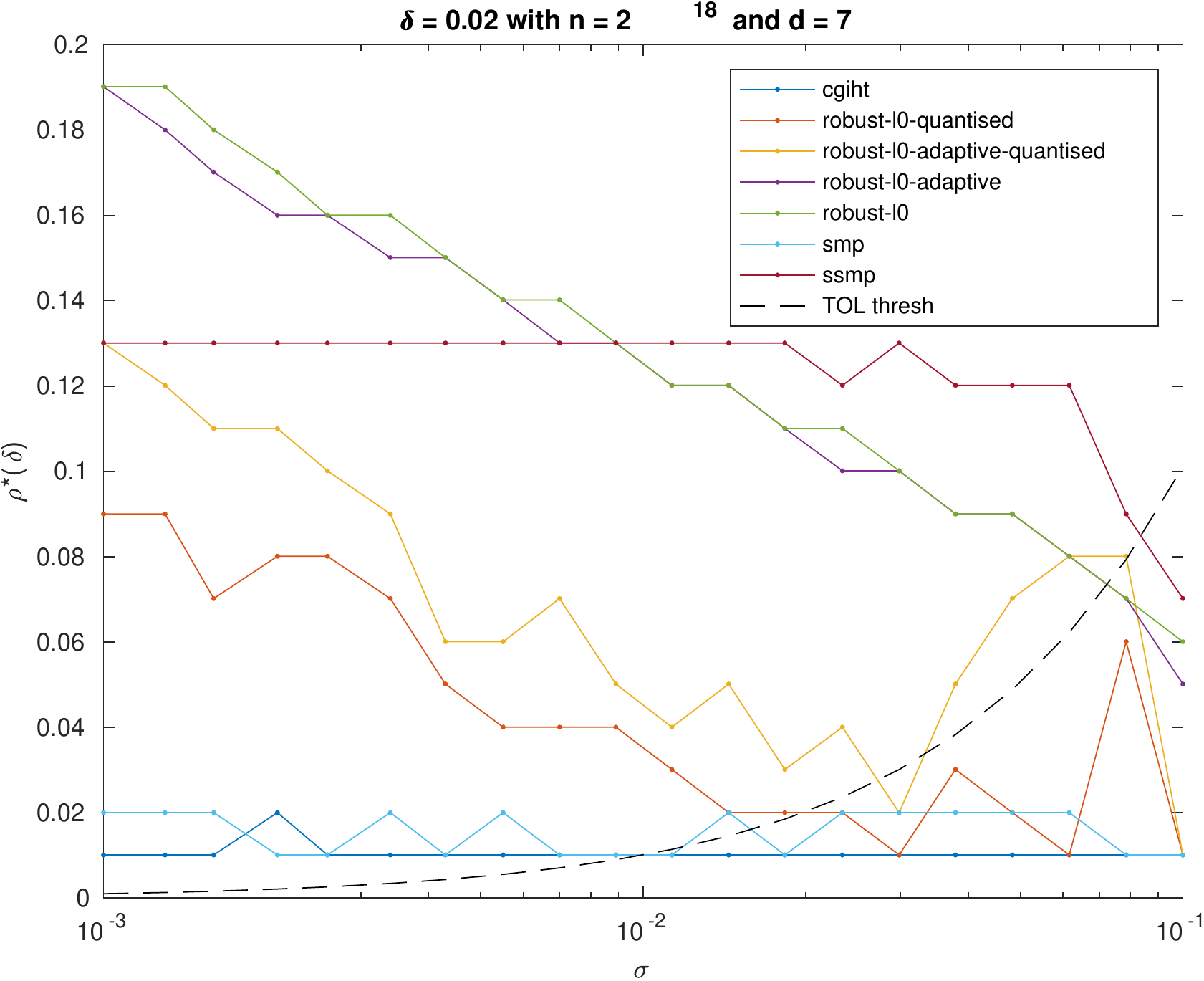}}\\
	\subfloat[$\delta=0.10$]{\includegraphics[width=0.49\linewidth]{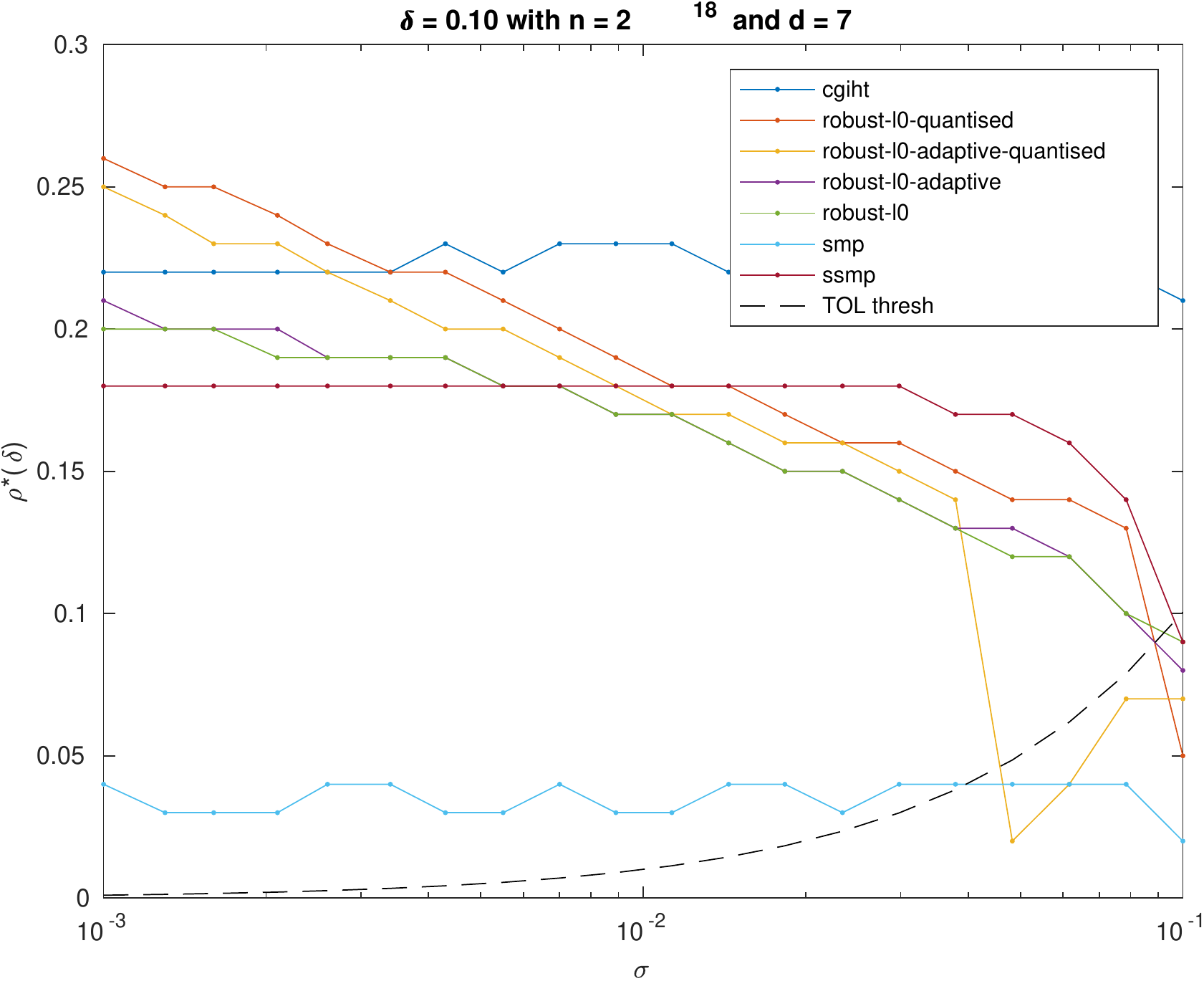}}
	\subfloat[$\delta=0.20$]{\includegraphics[width=0.49\linewidth]{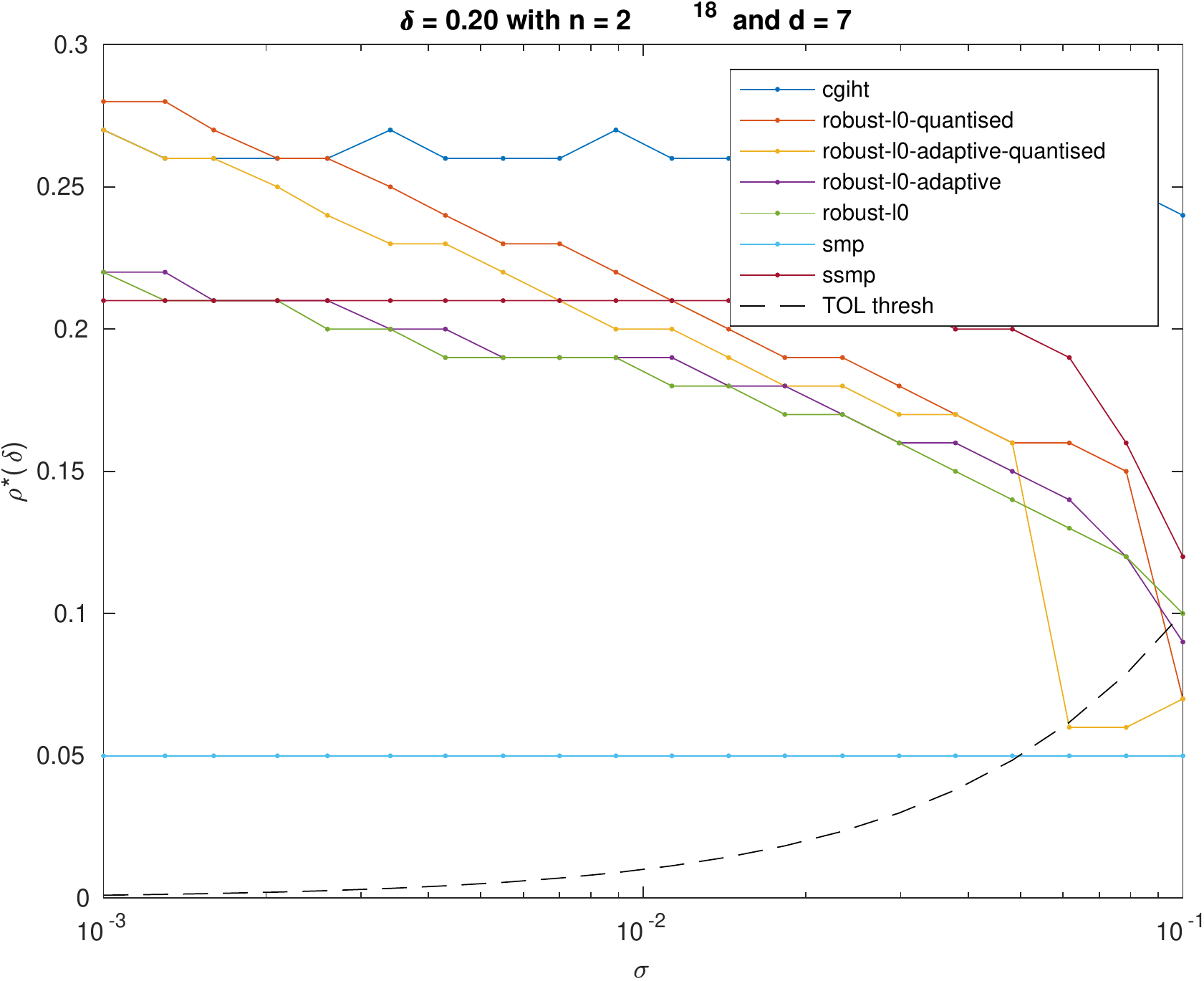}}
	\caption[Decrease in phase transition for varying $\sigma$.]{Decrease in phase transition for varying $\sigma$.}
	\label{fig:sigma_vs_rho_probs}
\end{figure}
We now investigate the extent to which the phase transitions of the algorithm decrease as we increase the noise level $\sigma$.
To do this, we consider $\delta \in \{0.01, 0.02, 0.1, 0.2\}$ and for each value of $\delta$ we compute we define the grid 
\[
\sigma \in \{10^{-3 + \frac{i}{10}} : i \in \{0\} \cup [20]\}.
\]
Then at each value of $\sigma$ we let $\rho = 0.01$ and draw ten problem instances from $\expandermodel$ with signal distribution $\mathcal{N}(0,1)$ and noise distribution $\mathcal{N}(0, \sigma^2)$.
If at least one of the problems was recovered successfully, then we set $\rho \leftarrow \rho + 0.01$ and repeat the experiment.
We do this procedure for each of the algorithms considered and record the largest $\rho$ having at least $50\%$ success rate.
The results are shown in Figure \ref{fig:sigma_vs_rho_probs}.
In order to show where the clipping in \eqref{eq:success_condition_clip} becomes active, the figures also show a {\tt TOL}-curve which partitions the space
into the region where the right hand side of
\eqref{eq:success_condition_clip} equals $\frac{1}{10}$ (bottom
region) and the region where it equals the right hand side of
\eqref{eq:success_condition} (top region).
We can appreciate from Figure \ref{fig:sigma_vs_rho_probs} that for small $\delta$ both Robust-$\ell_0$ and SSMP have the best recovery capabilities, with Robust-$\ell_0$ being preferable for noise levels $\sigma \lessapprox 10^{-2}$ and SSMP being better suited for larger noise levels.
For larger $\delta$, CGIHT is preferable except for very low noise levels.


\subsection{Phase transitions for extreme subsampling, $\delta\ll 1$}\label{subsec:delta}
The numerical experiments of Parallel-$\ell_0$ in \cite{Tanner:2015aa}
showed flat phase transitions; that is, it was observed that
$\rho^*(\delta)$ remained approximately $0.3$ as $\delta\rightarrow 0$
provided $n$ was sufficiently large.
While Robust-$\ell_0$ does not exhibit precisely the same behaviour in
the presence of noise, we do observe that $\rho^*(\delta)$ remains
nontrivial even for $\delta$ as small as $10^{-3}$, again provided $n$
is sufficiently large.
We provide numerical evidence for this in Figure \ref{fig:timing_low_delta}.
For each $(\delta, \sigma) \in \{0.001, 0.01\} \times \{0.001, 0.01\}$ we let $\rho = 0.01$ and solve ten problems drawn from $\expandermodel$ with nonzero distribution $\mathcal{N}(0, 1)$ and noise distribution $\mathcal{N}(0, \sigma^2)$. 
If at least one problem instance converges, we average the run-time of the problems that converged and repeat the process with $\rho \leftarrow \rho + 0.01$.
We plot the timing at each $\rho$ for parameter values for which at which at least $50\%$ of the problems were successfully recovered under the criteria \eqref{eq:success_condition_clip}.

It can be seen in Figure \ref{fig:tfd_a}-\ref{fig:tfd_d} that 
the phase transition either remains nearly unchanged or increases as
$n$ increases from $2^{22}$ to $2^{24}$.
%
%
In particular, Figure \ref{fig:tfd_a} shows that for $\sigma = 0.001$
and $\delta=0.001$, the phase transitions for all variants of the
Robust-$\ell_0$ algorithms in fact increase from $\rho\approx 0.08$ to
$\rho \approx 0.1$ as $n$ increases from $2^{22}$ to $2^{24}$.
Additionally, contrasting Figures \ref{fig:tfd_a} and \ref{fig:tfd_b}
or \ref{fig:tfd_c} and \ref{fig:tfd_d} shows that a ten-fold increase
in $\sigma$ has the expected effect of reducing the phase transition and increasing the
computational time.
Figures \ref{fig:tfd_a} and \ref{fig:tfd_b} show the phase
transition for Robust-$\ell_0$ and Robust-$\ell_0$-adaptive remain
significant even for $\delta$ as small as $10^{-3}$.
Figures \ref{fig:tfd_c} and \ref{fig:tfd_d} show results for the same
set of experiments, but for $\delta = 0.01$ which corresponds to a
ten-fold increase in $\delta$ over the value used in Figures
\ref{fig:tfd_a} and \ref{fig:tfd_b}. 
For $\sigma = 0.001$ the phase transition of Robust-$\ell_0$ reaches
$\rho \approx 0.17$, while for $\sigma = 0.01$ the phase transitions
drop to $\rho \approx 0.12$ and there is an increase in the
computational time.
\begin{figure}[!htbp]
	\centering
	\subfloat[$\delta = 0.001$, $\sigma = 0.001$]{\includegraphics[width=0.49\linewidth]{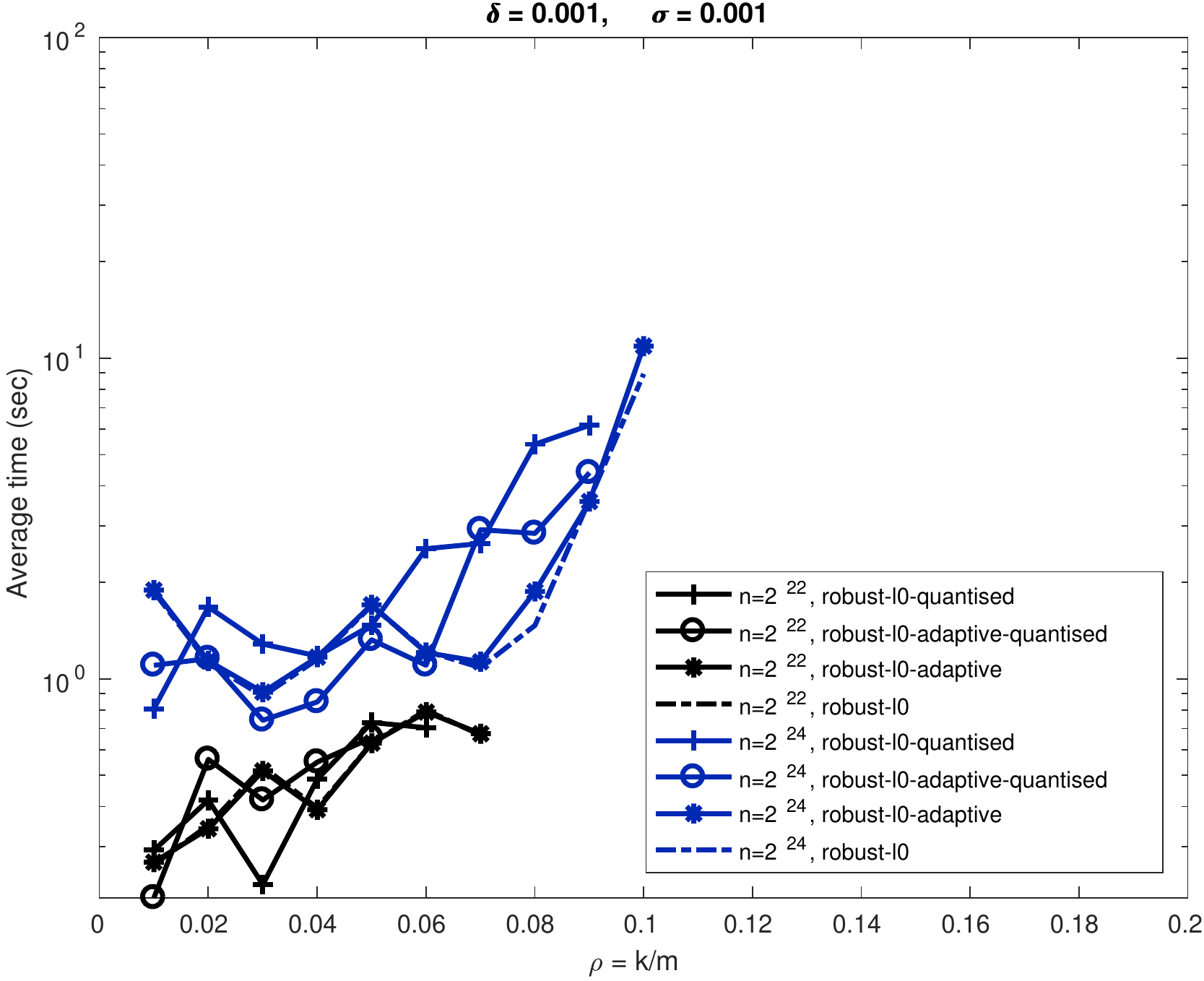} \label{fig:tfd_a}}
	\subfloat[$\delta = 0.001$, $\sigma = 0.01$]{\includegraphics[width=0.49\linewidth]{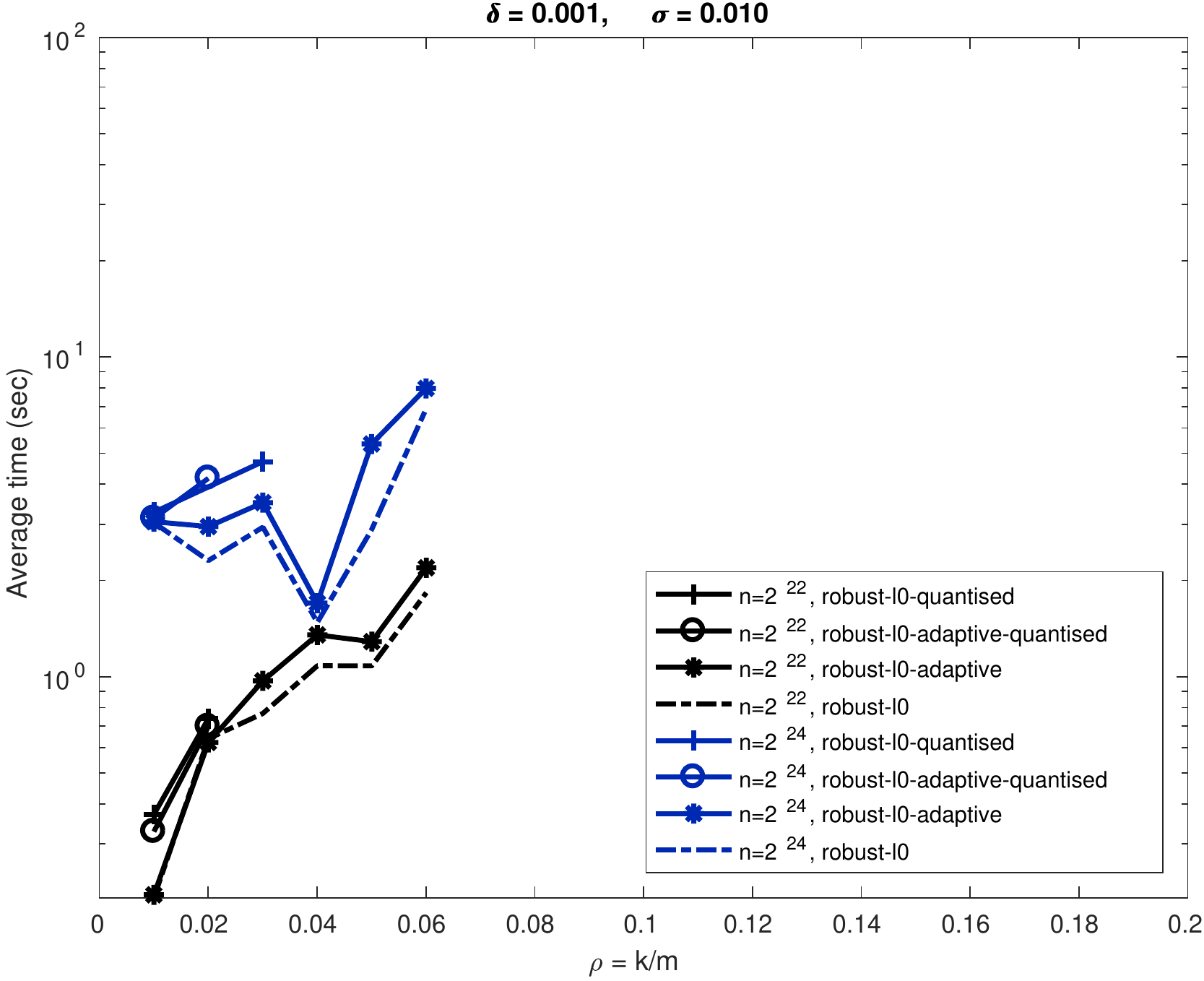} \label{fig:tfd_b}}

	\subfloat[$\delta = 0.01$, $\sigma = 0.001$]{\includegraphics[width=0.49\linewidth]{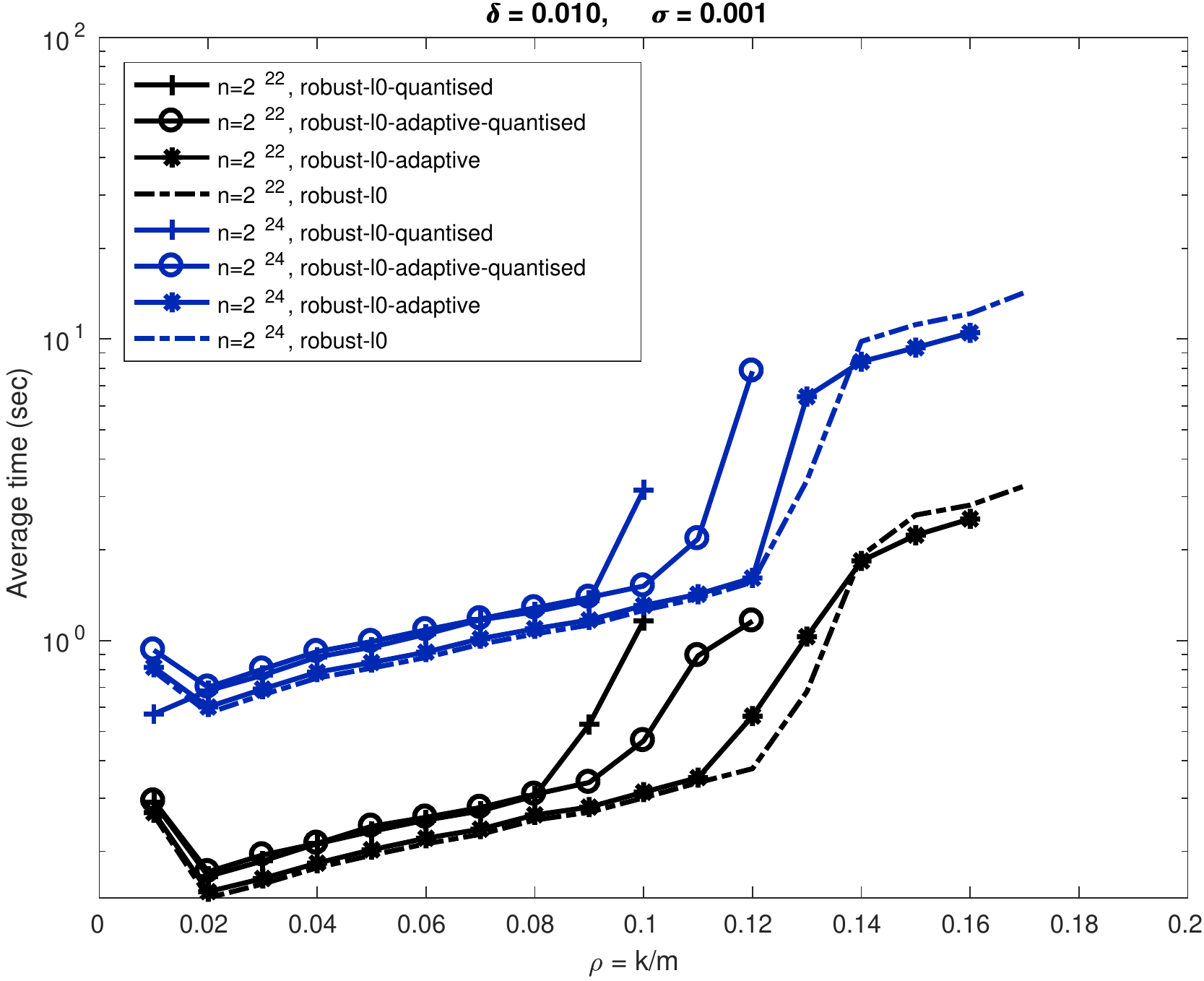} \label{fig:tfd_c}}
	\subfloat[$\delta=0.01$, $\sigma = 0.01$]{\includegraphics[width=0.49\linewidth]{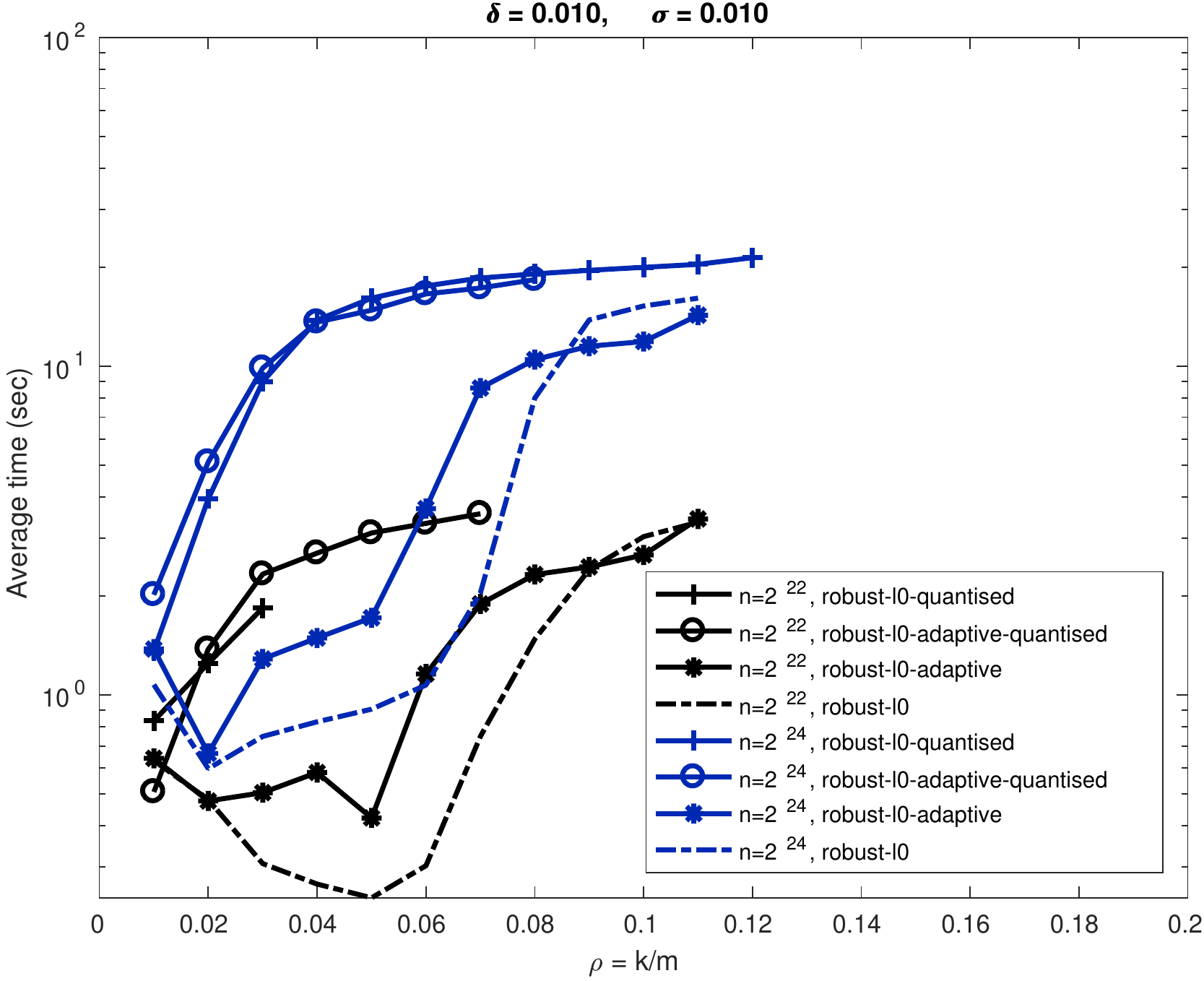} \label{fig:tfd_d}}
	\caption[Timing for low $\delta$]{Phase transitions and timing
          Robust-$\ell_0$ for $\delta\ll 1$.}
	\label{fig:timing_low_delta}
\end{figure}

\section{Conclusions}

We have shown that the decoding framework presented in \cite{Tanner:2015aa} can be extended to the case where the measurements are corrupted by additive noise.
This framework is extended by deriving the posterior distribution of an entry in the residual being zero or being equal to another residual entry given the corrupted measurements. 
This Bayesian approach to decoding was implemented in Robust-$\ell_0$ and its four variants.
We show that the resulting algorithms inherits some desirable properties from Parallel-$\ell_0$ like high phase transitions for low $\delta$ and large $n$ and low-latency.
However, these qualities are weakened by the corruption of the measurements. 
Our numerical experiments show that Robust-$\ell_0$ should be considered in cases of moderate noise and $\rho \lessapprox 0.3$.
\section*{Acknowledgments}

The authors would like to thank William Carson for many helpful discussions on imaging applications of this work.
This work was supported by The Alan Turing Institute under the EPSRC grant EP/N510129/1.


\begin{IEEEkeywords}
compressed sensing, expander graphs, dissociated signals, robust algorithms.
\end{IEEEkeywords}

\bibliographystyle{IEEEtran}
\bibliography{robust_l0.bib}

\end{document}